\documentclass{article}
\usepackage[utf8x]{inputenc}
\usepackage{amssymb}
\usepackage{amsmath}
\usepackage{mathrsfs}
\usepackage{bm}
\usepackage{amsthm}
\usepackage{enumitem}
\usepackage{color}
\usepackage{graphicx}
\usepackage{bm}
\usepackage{ulem}
\usepackage{upgreek}

\usepackage{accents}

\newcommand{\ci}[1]{\mathscr{#1}}
\newcommand{\g}[1]{\mathfrak{#1}}
\renewcommand{\ni}{\nu}
\newcommand{\alfa}{\alpha}
\newcommand{\R}{\mathbf{R}}
\newcommand{\C}{\mathbf{C}}
\newcommand{\de}{\partial}
\renewcommand{\H}{\mathbf{H}}
\newcommand{\N}[1]{\left\lVert#1\right\rVert}
\newcommand{\e}{\varepsilon}
\newcommand{\bra}{\left\langle}
\newcommand{\ket}{\right\rangle}

\renewcommand{\phi}{\varphi}
\newcommand{\con}[1]{\overline{#1}}

\newcommand{\mi}{\mu}

\newcommand{\cerchio}[1]{\accentset{\smash{\raisebox{-0.12ex}{$\scriptstyle\circ$}}}{#1}\rule{0pt}{2.3ex}}

\DeclareMathOperator{\Vol}{Vol}
\DeclareMathOperator{\ind}{ind}
\DeclareMathOperator{\Crit}{Crit}

\newtheorem{proposizione}{Proposition}[section]
\newtheorem{teorema}[proposizione]{Theorem}
\newtheorem{lemma}[proposizione]{Lemma}
\newtheorem{corollario}[proposizione]{Corollary}

\theoremstyle{definition}
\newtheorem{definizione}[proposizione]{Definition}

\theoremstyle{remark}

\title{Blow-up analysis and degree theory for the Webster curvature prescription problem in three dimensions}
\author{Claudio Afeltra
}
\date{}

\begin{document}

\maketitle

\begin{abstract}
 Given a strictly pseudoconvex CR manifold $M$ of dimension three and positive CR Yamabe class, and a positive smooth function $K:M\to\R$ verifying some mild and generic hypotheses, we prove the compactness of the set of solutions of the Webster curvature prescription problem associated to $K$, and we compute the Leray-Schauder degree in terms of the critical points of $K$. As a corollary, we get an existence result which generalizes the ones existent in the literature.
\end{abstract}

\section{Introduction}
Given a Riemannian manifold $(M,g)$ and a smooth function $K:M\to\R$, the problem of the existence (or non-existence) of a metric conformally equivalent to $g$ and whose scalar curvature is $K$ is of utmost importance in conformal geometry, and has been subject of a considerable number of studies.
The case of constant $K$, known as Yamabe problem, is the fundamental one and has been solved positively; in general there seems to be no complete classification of the functions $K$ which can be realised as scalar curvature by a conformal change, and various sufficient conditions and obstructions have been found (see, for example, Chapter 6 of \cite{Au}).

Since in CR geometry there exists a scalar invariant known as Webster curvature which shares many properties with the scalar curvature of Riemannian geometry, it is natural to study the problem of curvature prescription for Webster curvature.
Fixing a contact form $\theta$ then the condition that the form $\widetilde{\theta}=u^{\frac{2}{n}}\theta$ (with $u>0$) has Webster curvature $K$ is equivalent to the equation
\begin{equation}\label{EquazioneIntroduzione}
 L_{\theta}u = Ku^{\frac{n+2}{n}}
\end{equation}
where $L_{\theta}=-b_n\Delta_b+R$ with $b_n=2+\frac{2}{n}$ is the conformal sublaplacian (see Section \ref{SezioneNotazioni} for detailed definitions).

In order to solve this problem various techniques have been applied, for example variational methods (see for example \cite{G}), perturbation techniques (see \cite{MU}), geometric flows (see \cite{H}) and symmetry methods.

A widely used method in conformal geometry is to consider a sequence of solutions tending to infinity in a neighborhood of a point and rescaling it in normal coordinates in order to get a limit metric. This limit must be a conformal metric of constant scalar curvature in $\R^n$, which were completely classified by Caffarelli, Gidas and Spruck in \cite{CGS}, and turn out to be, up to translations and dilations, the spherical metric transferred to $\R^n$ through the stereographic projection. This method allows to understand how the sequence of solutions blows up, and is a fundamental step in many existence results for prescribed scalar curvature, as well as other problems.

An obstacle to do the same in CR geometry is that performing the same procedure one gets a contact form of constant Webster curvature on the Heisenberg group $\H^n$, and such forms have not yet completely classified.
The proof of Caffarelli, Gidas and Spruck in $\R^n$ cannot be adapted to $\H^n$ because it uses the method of moving planes, which requires reflecting solutions across a hyperplane, but $\H^n$ does not have CR automorphisms analogous to Euclidean reflections.

Jerison and Lee in \cite{JL2} proved a classification result valid only under the hypothesis of finite volume. Their result shows that under such hypothesis the solutions are, up to Heisenberg translations and dilations, the standard contact form of $S^{2n+1}$. This has allowed to employ in CR geometry blow up methods in situations where the hypothesis of finite volume can be proved, for example methods based on the change of topology of sublevels as in \cite{G}.

Recently Catino, Li, Monticelli and Roncoroni in \cite{CLMR} proved the classification theorem for $\H^1$. This allows to perform the blow up analysis in CR manifolds of dimension three (as, for example, has been made by the author of this article in \cite{Af}).

In this article we will employ blow up analysis for the problem of prescription of Webster curvature.
We will study the case in which $K>0$; by integrating Equation \eqref{EquazioneIntroduzione} it can easily be proved that this implies that the CR manifold has positive CR Yamabe class (see Section \ref{SezioneNotazioni} for the definition).
Under this hypothesis, it is known that the CR sublaplacian has a Green function $G_p$ at every point.
Furthermore, as will be explained in Section \ref{SezioneNotazioni}, there exists a contact form around $p$ which allows to define so-called CR normal coordinates, and the Green function with respect to such form in these coordinates has a principal regular term denoted by $A_p$.

Given a finite subset $S=\{\overline{x}^1,\ldots,\overline{x}^N\}$ let us define the matrix $M(S)$ by
\begin{equation}\label{DefinizioneMatrice}
 \begin{array}{l}
  M_{jj} = -\frac{\Delta_bK(\overline{x}^j)}{K(\overline{x}^j)^2} - 32\frac{A_{\overline{x}^j}}{K(\overline{x}^j)};\\
  M_{jk} = -\frac{32G_{\overline{x}^k}(\overline{x}^j)}{K(\overline{x}^k)^{1/2}K(\overline{x}^j)^{1/2}}
 \end{array}
\end{equation}
with respect to a contact form realising CR normal coordinates around each point of $S$ (which exists because $S$ is discrete).

We point out that since we are studying the problem of prescribed Webster curvature, performing a preliminary change of contact form does not affect the problem.

We will see that if $S$ is a set of critical point of $K$ then the matrix $M(S)$ governs the behavior of solution blowing up at each point of $S$, and in particular it describes the mutual interactions among the bubbles which form at the points of $S$. This behavior is typical of dimension three in CR geometry and of dimension four in Riemannian geometry: while in higher dimension the interaction among bubbles is negligible with respect to the interaction of the bubbles with $K$ and the manifold, and in lower dimension the opposite is true, in these dimensions the two interactions have the same order.
The relation between these two situations is connected to the fact that the homogeneous dimension of $2n+1$-dimensional CR manifolds is $2n+2$.
Indeed matrices similar to the ones defined in formula \eqref{DefinizioneMatrice} appear in the scalar curvature prescription problem in dimension four (see for example \cite{L} and \cite{BCCH}).

In our results we will need to require the mild and generic conditions that, calling $\Crit(K)$ the set of critical points of $K$, in CR normal coordinates
\begin{equation}\label{Condizione1}
 -\frac{\Delta_bK(\overline{x})}{K(\overline{x})} -32A_{\overline{x}}\ne 0 \;\;\;\;\;\;\forall\overline{x}\in\Crit(K)
\end{equation}
and furthermore, denoting by $\mi(M)$ the least eigenvalue of the matrix $M$, that, yet again in CR normal coordinates,
\begin{equation}\label{Condizione2}
 \mi(M(S))\ne 0 \;\;\;\;\;\;\forall S\subset\Crit(K)\;\;\text{such that}\;\;-\frac{\Delta_bK(\overline{x})}{K(\overline{x})} -32A_{\overline{x}}>0\;\;\;\;\forall\overline{x}\in S.
\end{equation}

\begin{teorema}\label{Teorema1}
 If $M$ is a three dimensional strictly pseudoconvex CR manifold of positive CR Yamabe class, $K>0$ is smooth, has isolated critical points,
 and verifies conditions \eqref{Condizione1} and \eqref{Condizione2},
 then the set
 $$\ci{M}_K =\left\{ u>0 \;\middle|\; L_{\theta}u = Ku^3 \right\}$$
 is compact in the $C^k$ topology for every $k$.
 Furthermore there exists $C>0$ such that $u>C$ for every $u\in\ci{M}_K$.
\end{teorema}

The proof of the above theorem is inspired by a similar result for four dimensional Riemannian manifolds by Li (see \cite{L}), and relies on many results about blow up from \cite{Af}.

Next we will prove a formula for the
total degree of the solution set.
We recall that the Morse index $\ind(K,\overline{x})$ of $K$ at $\overline{x}\in\Crit(K)$ is defined as the negative index of inertia of the Hessian of $K$ at $\overline{x}$ with respect to any Riemannian metric.
Furthermore we recall that for manifold with positive CR Yamabe class the conformal sublaplacian $L_{\theta}$ is invertible.

\begin{teorema}\label{Teorema2}
 In the hypotheses of Theorem \ref{Teorema1} let
 $$F(u) = u-L_{\theta}^{-1}(Ku^3)$$
 and let $C$ be such that $\ci{M}_K\subset\Omega_C$ where
 $$\Omega_C = \left\{ u\in \Gamma^{2,\alfa}(M) \;|\; \N{u}_{\Gamma^{2,\alfa}}<C, u>\frac{1}{C}\right\}$$
 (such $C$ exists thanks to Theorem \ref{Teorema1}).
 Then
 $$\deg(F,\Omega_C,0) = -1-\sum_{S\subseteq\Crit(K)\,:\,\mi(M(S))>0}(-1)^{\sum_{\overline{x}\in S}\ind(K,\overline{x})}.$$
\end{teorema}

As a corollary, thanks to the properties of the Leray-Schauder degree, we get an existence result which generalizes an analogous one from Gamara (\cite{G}) valid for spherical CR manifolds (that is, locally CR equivalent to $\H^1$).

\begin{corollario}
 In the hypotheses of Theorem \ref{Teorema1} if
 $$-1-\sum_{S\subseteq\Crit(K)\,:\,\mi(M(S))>0}(-1)^{\sum_{\overline{x}\in S}\ind(K,\overline{x})}\ne 0$$
 then there exists a contact form whose associated Webster is equal to $K$.
\end{corollario}

Theorem \ref{Teorema2} is proved by subcritical approximation: by defining
$$F_p(u) = u-L_{\theta}^{-1}(Ku^p)$$
in Section \ref{SezioneGrado} we will prove, using an homotopy with between $K$ and a constant function and a homotopy between the functional and a modified functional which has bounded set of zeroes uniformly in $p\in[1,1+\e]$, that $\deg(F_p,\Omega_R,0)=-1$ if $R$ is big enough such that $F_p^{-1}(0)\subset\Omega_R$ by reducing it to a linear problem.
Then in Section \ref{SezioneSoluzioniApprossimate} we will build, for every $S\subseteq\Crit(K)$ such that $M(S)>0$, a family of approximate solutions of the subcritical problem $F_p=0$ blowing up at $S$ as a linear combination of bubbles in local coordinates, and we will prove that there exist $\tau_0>0$ and $C$ such that for $p\in(3-\tau_0,3)$
$$F_p^{-1}(0) \subset \Omega_C\cup\bigcup_{S\subseteq\operatorname{Crit}(K):M(S)>0}U_{\tau}(S)$$
where $U_{\tau}(S)$ is a certain neighborhood of the set of the family of approximate solutions.
Since, by the properties of the Leray-Schauder degree
$$\deg(F_p,\Omega_R,0)= \deg(F_p,\Omega_C,0) + \sum_{S\subseteq\operatorname{Crit}(K):M(S)>0}\deg(F_p,U_{\tau}(S),0)$$
and $\deg(F_p,\Omega_C,0)=\deg(F,\Omega_C,0)$, in order to compute $\deg(F,\Omega_C,0)$ we have to compute $\deg(F_p,U_{\tau}(S),0)$. This will be achieved through various delicate estimates (see \cite{L} and \cite{MM} for similar computations in Riemannian geometry).

\paragraph{Acknowledgements.} The author is supported by the fund “MIUR PRIN 2017” - CUP:E64I19001240001.

\section{Preliminaries and notation}\label{SezioneNotazioni}
Let $M$ be a nondegenerate three dimensional CR manifold, that is, a three dimensional manifold endowed with a one dimensional subbundle $\ci{H}\subset TM\otimes\C$ such that $\ci{H}\cap\overline{\ci{H}}=\{0\}$ and $H(M)=\g{Re}(\ci{H}\oplus\overline{\ci{H}})$ is a contact distribution.
For an introduction to such manifolds we refer to \cite{DT}.

The choice of a contact form $\theta$ for the CR structure determines a rich geometric structure, including a nondegenerate symmetric tensor $G_{\theta}$ on $H(M)$, and a connection known as Tanaka-Webster connection. Contracting the curvature tensor of the connection twice using $G_{\theta}$, analogously to Riemannian geometry, it can be defined a scalar invariant $R_{\theta}$ known as Webster curvature.
The CR manifold is called strictly pseudoconvex if $G_{\theta}$ is positive definite, and thus a subriemannian metric. $G_{\theta}$ can be extended to a Riemannian metric $g_{\theta}$.

As mentioned in the introduction, since the contact form is determined up to the multiplication by a nowhere zero function, similarly to conformal geometry, given a function $K$ on $M$ it is natural to study whether there exists a contact form such that $R_{\theta}=K$. This is known as Webster curvature prescription problem.
When $K$ is a constant, the problem is known as CR Yamabe problem; analogously to the Riemannian case, the solution of the problem is influenced by the value of the Yamabe constant
$$\ci{Y}(M,\theta)=\inf_{\widetilde{\theta}\in[\theta]}\frac{\int_M R_{\widetilde{\theta}}dV_{\widetilde{\theta}}}{\operatorname{Vol}_{\widetilde{\theta}}(M)^{1/2}}.$$
We denote by $\nabla_b$ the subriemannian gradient, for which the Sobolev inequality
\begin{equation}\label{Sobolev}
 \N{u}_{L^4(M)} \le C \N{\nabla_bu}_{L^2(M)}
\end{equation}
holds if $M$ is compact (see Proposition 5.5 in \cite{JL1}).
We call $\Delta_b=\operatorname{div}\circ\nabla_b$, where the divergence is respect to $g_{\theta}$, the sublaplacian operator. The sublaplacian is not elliptic, but has many properties in common with elliptic operators, such as regularity theorems, maximum principles and the Harnack inequality (see Section 2.2 in \cite{Af} for a synthetic recall and appropriate references).

The most important three-dimensional CR manifold is the Heisenberg group $\H^1=\C\times\R$, which is the Lie group under the law $(z,t)\cdot(w,s) = (z+w,t+s+2\g{Im}(z\con{w}))$, endowed with the CR structure generated by the left invariant vector field
$$Z= \frac{\de}{\de z} +i\overline{z}\frac{\de}{\de t}$$
and the left invariant contact structure $\theta= dt+izd\con{z}-i\con{z}dz$.

On $\H^1$ we will use also the real left invariant vector fields
$$X= \frac{\de}{\de x} +2y\frac{\de}{\de t}, \;\;\; Y= \frac{\de}{\de y} -2x\frac{\de}{\de t}$$
and the right invariant vector field
\begin{equation}\label{CampoInvarianteADestra}
 Z_r= \frac{\de}{\de z} -i\overline{z}\frac{\de}{\de t}
\end{equation}
which verifies $[Z,Z_r]=0$, since left invariant vector fields on a Lie group commute with right invariant ones.
Furthermore we will use the Korányi norm $|(z,t)|=\left(|z|^4+t^2\right)^{\frac{1}{4}}$.

$\H^1$ has a one parameter family of automorphisms called dilations, defined as
$$\delta_{\lambda}(z,t)=(\lambda z,\lambda^2t)$$
for $\lambda>0$, which has many properties in common with dilations in $\R^n$.
The generator of the dilations is the vector field
$$\Xi = \sum_{\alfa=1}^n\left(z_{\alfa}Z_{\alfa}+\con{z}_{\alfa}Z_{\con{\alfa}}\right) +2tT.$$

Fixing a contact form, in the neighborhood of every point canonical coordinates with values in $\H^1$, analogous to the normal coordinates of Riemannian geometry, are defined, known as pseudohermitian normal coordinates (see \cite{JL3}). Furthermore in a neighborhood of every point there exists a contact form such that the pseudohermitian structure is ``as flat as possible''. The pseudohermitian normal coordinates with respect to this form are called CR normal coordinates (see \cite{JL3}), and are analogous to the conformal normal coordinates introduced by Lee and Parker (see \cite{LP}).
For the properties of CR and pseudohermitian normal coordinates we refer to \cite{JL3}, or to \cite{CMY1} for the case of dimension three.

The conformal sublaplacian is defined as $L_{\theta}=-4\Delta_b+R$.
When the Yamabe class $\ci{Y}(M)$ is positive, $L_{\theta}$ is positive definite, and has a Green function $G_p(x)$.
By Proposition 5.3 in \cite{CMY1}, $G_p(x)$ has the following expansion in CR normal coordinates near $p$:
$$ G_p(x)= \frac{1}{4\pi|x|^2} + A_p + w(x)$$
with $w(x)=o(1)$, where $A_p$ is a constant.

On $\H^1$ the equation $L_{\theta}u = u^3$ has the Jerison-Lee solution
\begin{equation}\label{BollaFormula}
 U=\frac{1}{\left(t^2+(1+|z|^2)^2\right)^{\frac{1}{2}}}.
\end{equation}

\paragraph{Leray-Schauder degree} Let $X$ be a Banach space, $\Omega\subseteq X$ an open bounded subset thereof, and let $T:\overline{\Omega}\to X$ be a (non necessarily linear) compact operator, $S=I-T$, and $x\in X\setminus S(\de\Omega)$.
Then the Leray-Schauder degree $\deg(S,\Omega,x)$ is a generalization of the topological degree for maps between finite dimensional vector spaces, which can be defined through finite dimensional approximation. We refer to Section 3.4 in \cite{AM} for the construction and the main properties of theis topological invariant.

\section{Blow-up analysis}
Let $M$ be a three dimensional pseudoconvex pseudohermitian CR manifold of positive CR Yamabe class, and let $K$ be a positive function of class $C^2$ on $M$.
We recall some definitions about blow-ups of sequences of solutions.
Let $p_i\in(1,3]$ such that $p_i\to 3$ and let $u_i$ be a solution of
\begin{equation}\label{EquazioneBlowUp}
 L_{\theta}u_i = K u_i^{p_i}.
\end{equation}

\begin{definizione}
 \begin{itemize}
  \item A point $\overline{x}\in M$ is called a blow-up point if there exist a sequence $x_i\to\overline{x}$ such that $M_i=u_i(x_i)\to\infty$.
  \item A point $\overline{x}\in M$ is called an isolated blow-up point if there exist $\overline{r}>0$, a constant $C$, and a sequence $x_i\to\overline{x}$ such that $x_i$ is a local maximum of $u_i$, $u_i(x_i)\to\infty$, and
  $$u_i(x)\le Cd(x,x_i)^{-\frac{2}{p_i-1}}$$
  for every $x\in B_{\overline{r}}(x_i)$.
  \item An isolated blow-up point $\overline{x}$ is called an isolated simple blow-up point $\overline{x}$ if there exists $\rho\in(0,\overline{r})$ (independent of $i$) such that, defining
  $$\overline{u}_i(r) = \int_{\de B_1(x_i)}u_i\circ\delta_rd\cerchio{\sigma}$$
  in pseudohermitian normal coordinates, and $\overline{w}_i(r)= r^{\frac{2}{p_i-1}}\overline{u}_i(r)$, $\overline{w}_i$ has exactly one critical point in $(0,\rho)$.
 \end{itemize}
\end{definizione}

The following results were proved in \cite{Af}.

\begin{proposizione}\label{BlowUpIsolatiBolle}
 If $\overline{x}$ is an isolated blow-up point then for any $R_i\to\infty$, $\e_i\to 0$ and $k\in\mathbf{N}$, up to subsequences, in pseudohermitian normal coordinates around $\overline{x}$ it holds that
 $$\N{\frac{1}{M_i}u_i\left(\delta_{M_i^{-\frac{p_i-1}{2}}}(x_i^{-1} \cdot x)\right) -(U\circ\delta_{K(0)^{1/2}/2})(x)}_{\Gamma^{k,\alfa}(B_{R_i}(0))}\le\e_i,$$
 where $U$ is defined in formula \eqref{BollaFormula}, $M_i=u_i(x_i)$, and
 $$\frac{R_i}{\log M_i} \to 0.$$
\end{proposizione}

\begin{proposizione}\label{LimiteBlowUp}
 If $\overline{x}$ is an isolated simple blow-up point then there exists $C$ such that
 $$M_iu_i(x)\le C d(x,x_i)^{-2}$$
 if $d(x,x_i)\le\frac{\rho}{2}$.
 Furthermore, up to subsequences, there exists $a>0$ such that
 $$M_iu(x)\to aG_{\overline{x}}(x) + b$$
 in $C^2_{\mathrm{loc}}(B_{\frac{\rho}{2}}(\overline{x})\setminus\{\overline{x}\})$, where $G_{\overline{x}}$ is the Green function of $L_{\theta}$ (which exists because $M$ has positive CR Yamabe class)
 and $L_{\theta}b=0$ on $B_{\frac{\rho}{2}}(\overline{x})$.
\end{proposizione}

\begin{proposizione}\label{StimeCitazione}
 In the above hypotheses, if $\tau_i=3-p_i$ then $\tau_i=O(M_i^{-2})$ and in particular $M_i^{\tau_i}\to 1$.
 Furthermore $u_i$ satisfies the following estimates in $B_{\rho}(x_i)\setminus B_{r_i}(x_i)$, where $r_i=R_iM_i^{-\frac{p_i-1}{2}}$:
 $$u_i(x)\le C M_i^{-1}d(x,x_i)^{-2},$$
 $$|(u_i)_{,1}(x)|\le C M_i^{-1}d(x,x_i)^{-3},$$
 $$|(u_i)_{,11}(x)|\le C M_i^{-1}d(x,x_i)^{-4}$$
 $$|(u_i)_{,1\con{1}}(x)|\le C M_i^{-1}d(x,x_i)^{-4},$$
 $$|(u_i)_{,0}(x)|\le C M_i^{-1}d(x,x_i)^{-4}$$
 $$|(u_i)_{,01}(x)|\le C M_i^{-1}d(x,x_i)^{-5},$$
 $$|(u_i)_{,00}(x)|\le C M_i^{-1}d(x,x_i)^{-6}.$$
 The same estimates hold for the corresponding derivatives in $\H^1$ in pseudohermitian normal coordinates.
\end{proposizione}

\begin{proof}
 The estimate on $\tau_i$ is proved as Lemma 4.9 in \cite{Af} but using Proposition 4.11 there instead of Lemma 4.6.
 The statement about the derivatives in pseudohermitian normal coordinates follows from Lemmas A.4 in \cite{Af}.
\end{proof}

In order to perform the blow-up analysis, we will need an estimate for $\nabla_bK(x_i)$.

\begin{proposizione}\label{GradienteNullo}
 If $x_i\to\overline{x}$ is an isolated simple blow-up point then $\nabla_b K(x_i)=O(M_i^{-1})$ and $TK(x_i)=O(M_i^{-1})$. In particular $\nabla_{g_{\theta}} K(\overline{x})=0$.
\end{proposizione}

\begin{proof}
 Let $r_i=R_iM_i^{-\frac{p_i-1}{2}}$ and let $\chi$ be a smooth function supported in $B_{\rho}(x_i)$ such that $\chi=1$ in $B_{\rho/2}(x_i)$ and $0\le\chi\le 1$.
 Let $Z_r$ be the right invariant vector field on $\H^1$ defined in formula \eqref{CampoInvarianteADestra}. Since $Z_r=Z-2i\con{z}T$, the estimates stated in Proposition \ref{StimeCitazione} for $\cerchio{Z}u_i$ hold also for $\cerchio{Z}_ru_i$.
 
 Using pseudohermitian normal coordinates around $x_i$, let us multiply equation \eqref{EquazioneBlowUp} by $\cerchio{Z}_ru_i\chi$ and integrate with respect to $d\cerchio{V}$. Using Lemmas \ref{StimaDifferenzaSublaplaciano}, \ref{LemmauZru} and \ref{StimaIntegralex1u4}, the fact that $[\cerchio{Z},\cerchio{Z}_r]=0$ and Propositions \ref{BlowUpIsolatiBolle} and \ref{StimeCitazione}, we get
 $$0=\int_{B_{\rho}(x_i)}(-4\Delta_bu_i+ Ru_i)\cerchio{Z}_ru_i\chi d\cerchio{V} -\int_{B_{\rho}(x_i)} Ku^{p_i}\cerchio{Z}_ru_i\chi d\cerchio{V} =$$
 $$= -4\int_{B_{\rho}(x_i)}\cerchio{\Delta}_bu_i\cerchio{Z}_ru_i\chi d\cerchio{V}+ O(M_i^{-1})+\int_{B_{\rho}(x_i)}Ru_i\cerchio{Z}_ru_i\chi d\cerchio{V}+$$
 $$-\frac{1}{p_i+1}\int_{B_{\rho}(x_i)}\left(K(x_i)+ZK(x_i)z + \con{Z}K(x_i)\con{z} + O(|x|^2)\right)\cerchio{Z}_r(u_i^{p_i+1})\chi d\cerchio{V} =$$
 $$= 4\int_{B_{\rho}(x_i)}\cerchio{Z}u_i\cerchio{\con{Z}}\cerchio{Z}_ru_i\chi d\cerchio{V}+ 4\int_{B_{\rho}(x_i)}\cerchio{\con{Z}}u_i\cerchio{Z}\cerchio{Z}_ru_i\chi d\cerchio{V}+O(M_i^{-1})+$$
 $$+O(M_i^{-1})+\frac{1}{p_i+1}\int_{B_{\rho}(x_i)}\left(Z_rK(x_i)+ O(|x|)\right)u_i^{p_i+1}\chi d\cerchio{V} +$$
 $$+\frac{1}{p_i+1}\int_{B_{\rho}(x_i)}\left(K(x_i)+ZK(x_i)z + \con{Z}K(x_i)\con{z} + O(|x|^2)\right)u_i^{p_i+1}\cerchio{Z}_r(\chi)d\cerchio{V} =$$
 $$= 4\int_{B_{\rho}(x_i)}\cerchio{Z}_r|\cerchio{Z}u_i|^2\chi d\cerchio{V}+O(M_i^{-1})+$$
 $$+\frac{1}{p_i+1}ZK(x_i)M_i^{-2(p_i-1)}\int_{B_{R_i}(x_i)}(u_i\circ\delta_{M_i^{-\frac{p_i-1}{2}}})^{p_i+1}\chi\circ\delta_{M_i^{-\frac{p_i-1}{2}}} d\cerchio{V}+$$
 $$+ZK(x_i)O\left(\int_{B_{\rho}(x_i)\setminus B_{r_i}(x_i)}M_i^{-(p_i+1)}|x|^{-2(p_i+1)}d\cerchio{V}\right)+$$
 $$+O\left(\int_{B_{\rho}(x_i)}|x|u_i^{p_i+1}\chi d\cerchio{V}\right)+O\left(\int_{B_{\rho}(x_i)\setminus B_{\rho/2}(x_i)}M_i^{-(p_i-1)}|x|^{-2(p_i-1)}d\cerchio{V}\right) =$$
 $$= -4\int_{B_{\rho}(x_i)}|\cerchio{Z}u_i|^2\cerchio{Z}_r\chi d\cerchio{V}+O(M_i^{-1})+(C+o(1))ZK(x_i)+$$
 $$+ZK(x_i)O\left(M_i^{-4}r_i^{4-2(p_i+1)}\right)+O\left(M_i^{-1}\right)+O\left(M_i^{-2}\right) =$$
 $$= O\left(\int_{B_{\rho}(x_i)\setminus B_{\rho/2}(x_i)}M_i^{-2}|x|^{-6}d\cerchio{V}\right) +O\left(M_i^{-1}\right)+(C+o(1))ZK(x_i)+$$
 $$+ZK(x_i)O\left(R_i^{-4}\right)+O\left(M_i^{-1}+M_i^{-4}r_i^{5-2(p_i+1)}\right)+O\left(M_i^{-2}\right) =$$
 $$= O\left(M_i^{-1}\right) + (C+o(1))ZK(x_i)$$
 which implies the first part of the thesis. The estimate for $TK(x_i)$ is proved similarly by multiplying equation \eqref{EquazioneBlowUp} by $Tu_i\chi$.
\end{proof}

In order to carry out the blow-up analysis, we will need a version of the above Proposition for the case in which also $K$ depends on $i$.

\begin{lemma}
 If $u_i$ is a sequence of solutions of the equation
 $$L_{\theta}u_i = K_iu_i^{p_i}$$
 with $K_i$ satisfying $\N{\nabla_{g_{\theta}}K_i}_{L^{\infty}}\to 0$, $\N{\nabla^2_{g_{\theta}}K_i}_{L^{\infty}}\to 0$, and if $x_i\to\overline{x}$ is an isolated simple blow-up point then $\nabla_b K(x_i)=o(M_i^{-1})$.
\end{lemma}

\begin{proof}
 The proof is almost equal to the proof of Proposition \ref{GradienteNullo}, but using CR normal coordinates, and using the hypotheses to write $Z_rK_i(x)= Z_rK_i(x_i)+o_i(1)O(|x|)$. The necessary estimates from Proposition \ref{StimeCitazione} are extended to the case of $K_i$ depending by $i$ with our hypotheses in a straightforward way.
\end{proof}

\begin{lemma}\label{SegnoTermineOrdineZero}
 If $\overline{x}$ is an isolated simple blow-up point for the equation
 $$L_{\theta}u_i = K_iu_i^{p_i}$$
 with $K_i$ satisfying $\N{\nabla_{g_{\theta}}K_i}_{L^{\infty}}\to 0$, $\N{\nabla^2_{g_{\theta}}K_i}_{L^{\infty}}\to 0$, and $M_iu_i\to h$ in $C^2_{\mathrm{loc}}(B_r(\overline{x})\setminus\{\overline{x}\})$, and in CR normal coordinates around $\overline{x}$ one has $h(x)=\frac{a}{|x|^2}+ A + o(1)$ with $a>0$, then $A\le 0$.
\end{lemma}

\begin{proof}
 The proof is the same of Lemma 4.14 in \cite{Af}, except that the term in the Pohozaev identity involving the derivative of $K_i$ is different due to the fact that $K_i$ is not constant:
 $$\frac{1}{4}\frac{1}{p_i+1}\Xi((\phi^{-\tau_i}K_i)\circ\delta_{\sigma})(\widetilde{u}_i\circ\delta_{\sigma})^{p_i+1}=$$
 $$= \frac{1}{4}\frac{1}{p_i+1}\Xi((\phi\circ\delta_{\sigma})^{-\tau_i})K_i\circ\delta_{\sigma}(\widetilde{u}_i\circ\delta_{\sigma})^{p_i+1}+ \frac{1}{4}\frac{1}{p_i+1}(\phi\circ\delta_{\sigma})^{-\tau_i}\Xi(K_i\circ\delta_{\sigma})(\widetilde{u}_i\circ\delta_{\sigma})^{p_i+1}.$$
 where, we recall, $\widetilde{u}_i=\frac{u_i}{\phi}$, $\phi$ being the conformal factor defining CR normal coordinates.
 The first term is treated in the same way as the original proof, so we must estimate the second one: for $i\to\infty$, fixing $\sigma\in(0,1]$ and using Lemma \ref{StimaIntegralex1u4}
 $$\left|\int_{B_1(x_i)}(\phi\circ\delta_{\sigma})^{-\tau_i}\Xi(K_i\circ\delta_{\sigma})(\widetilde{u}_i\circ\delta_{\sigma})^{p_i+1}\right| \lesssim \left|\int_{B_1(x_i)}\Xi(K_i\circ\delta_{\sigma})(u_i\circ\delta_{\sigma})^{p_i+1}\right| \le$$
 $$\le \left|\int_{B_1(x_i)}|\nabla_bK_i(x_i)||x|\circ\delta_{\sigma}(u_i\circ\delta_{\sigma})^{p_i+1}\right| +$$
 $$+O\left(\left(TK_i(x_i)+\N{\nabla^2_{g_{\theta}}K_i}_{L^{\infty}}\right)\int_{B_1(x_i)}|x|^2\circ\delta_{\sigma}(u_i\circ\delta_{\sigma})^{p_i+1}\right) =$$
 $$= o(M_i^{-1})O(M_i^{-1}) + o(1)O(M_i^{-2})= o(M_i^{-2})$$
 so taking $\limsup_{i\to\infty}M_i^2$ to the required term one gets zero. Therefore applying $\lim_{\sigma\to 0}\sigma^3\limsup_{i\to\infty}M_i^2$ as in the original proof the further term arising vanishes as needed.
\end{proof}

Thanks to the above lemma, the blow-up analysis like the one carried out in \cite{Af} may be performed with minor modifications, so the following Theorem may be proved.
Notice that the hypothesis that $\N{\nabla_{g_{\theta}}K_i}_{L^{\infty}}\to 0$ and $\N{\nabla^2_{g_{\theta}}K_i}_{L^{\infty}}\to 0$ is verified when needed because the analogous result in \cite{Af}, Lemma 4.14, is applied twice, in the proofs of Propositions 4.15 and 4.17, both times to sequences of functions of the type $\xi_i=\mi_i^{\frac{2}{p_i-1}}\widetilde{u}_i\circ\delta_{\mi_i}$ with $\mi_i\to 0$, which in this case verifies the equation
$$L_{\theta_i}\xi_i= K\circ\delta_{\mi_i}(\phi_i\circ\delta_{\mi_i})^{-\tau_i}\xi_i^{p_i}$$
and the sequence $K_i=K\circ\delta_{\mi_i}$ satisfies the required hypothesis.

\begin{teorema}\label{TeoremaBlowUpIsolati}
 Given a sequence $u_i$ of solutions of equation \eqref{EquazioneBlowUp}, up to subsequences there exists a finite set of blow-up points, and they are isolated simple blow-up points in CR normal coordinates.
\end{teorema}

Let $S=\{\overline{x}^1,\ldots,\overline{x}^N\}$ the set of blow-up points, and $x_i^k\to\overline{x}^k$ sequences of local maxima.
Using Proposition \ref{LimiteBlowUp} and using the maximum principle and the Harnack inequality (as, for example, in the proof of Theorem 1.1 in \cite{Af}), we have for every $k$
$$u_i(x_i^k)u_i \to h_k=\sum_{j=1}^Na_k^jG_{\overline{x}^j} +b_k$$
in $C^2(M\setminus S)$, where $G_x$ is the Green function for $L_{\theta}$,
where $a_k^j>0$, $b_k\in C^2(M)$ and $L_{\theta}b_k=0$. Since $M$ has positive CR Yamabe class, $L_{\theta}$ has trivial kernel, and so $b_k=0$ and
\begin{equation}\label{LimiteFunzioneGreen}
 u_i(x_i^k)u_i \to h_k=\sum_{j=1}^Na_k^jG_{\overline{x}^j}
\end{equation}
in $C^2(M\setminus S)$.

\begin{lemma}
 $$a_k^k= \frac{32\pi}{K(\overline{x}^k)}.$$
\end{lemma}

\begin{proof}
 Integrating the function $-\Delta_b(u_i(x_i^k)\sigma^{\frac{4}{p_i-1}}u_i\circ\delta_{\sigma})$ in
 $B_r(\overline{x}^k)$
 we get
 $$-\int_{B_{\rho}(\overline{x}^k)}\Delta_b(u_i(x_i^k)\sigma^{\frac{4}{p_i-1}}u_i) = -\int_{\de B_{\rho}(\overline{x}^k)}\nabla_b(u_i(x_i^k)\sigma^{\frac{4}{p_i-1}}u_i)\cdot\ni =$$
 $$=-\int_{\de B_{\rho}(\overline{x}^k)}\nabla_b\left(\sigma^{\frac{4}{p_i-1}}\sum_{j=1}^Na_k^jG_{\overline{x}^j}\circ\delta_{\sigma}\right)\cdot\ni +\sigma^{\frac{4}{p_i-1}}o(1).$$
 Taking the limit for $i\to\infty$ and then for $\sigma\to 0$ one gets
 $$-a_k^k\int_{\de B_1(\overline{x}^k)}\nabla_b\left(\frac{1}{4\pi|x|^2}\right)\cdot\ni
 $$
 and since $\frac{1}{\pi|x|^2}$ is the fundamental solution of $-\Delta_b$ on $\H^1$, considering a sequence of smooth functions converging weakly to it it can be proved that this is equal to $\frac{a_k^k}{4}$.
 
 On the other hand, using equation \eqref{EquazioneBlowUp}, and calling, for every $k$, $R_i^k$ and $r_i^k$ the distances defined in Propositions \ref{BlowUpIsolatiBolle} and \ref{StimeCitazione},
 $$-4\int_{B_{\rho}(\overline{x}^k)}\Delta_b(u_i(x_i^k)\sigma^{\frac{4}{p_i-1}}u_i\circ\delta_{\sigma}) =$$
 $$=\sigma^2\sigma^{\frac{4}{p_i-1}}u_i(x_i^k)\int_{B_{\rho}(\overline{x}^k)}\left(-R\circ\delta_{\sigma}u_i\circ\delta_{\sigma} + K\circ\delta_{\sigma}\left(u_i\circ\delta_{\sigma}\right)^{p_i}\right) =$$
 $$= \sigma^{2+\frac{4}{p_i-1}}u_i(x_i^k)u_i(x_i^k)^{-2(p_i-1)}\sigma^{-4}\int_{B_{\sigma R_i^k}(\overline{x}^k)}\left(O\left(u_i(x_i^k)(1+|x|)^{-2}\right) +\right.$$
 $$\left.+(K(\overline{x}^k)+\sigma O(|x|))\left(u_i(x_i^k)U\circ\delta_{K(\overline{x}^k)^{1/2}/2}\right)^{p_i} \right)+$$
 $$+ O\left(\sigma^{2+\frac{4}{p_i-1}}u_i(x_i^k)\int_{B_{\rho}(\overline{x}^k)\setminus B_{r_i^k}(\overline{x}^k)}\sigma^{-2}u_i(x_i^k)^{-2}|x|^{-2} + (\sigma^{-2}u_i(x_i^k)^{-2}|x|^{-2})^{p_i}\right)$$
 Taking the limit for $i\to\infty$ and then for $\sigma\to 0$ and using Lemma \ref{IntegraleU3} one gets
 $$K(\overline{x}^k)\int_{\H^1}\left(U\circ\delta_{K(\overline{x}^k)^{1/2}/2}\right)^3 = K(\overline{x}^k)16K(\overline{x}^k)^{-2}\int_{\H^1}U^3= 32
 \pi\frac{1}{K(\overline{x}^k)}$$
 so we get the thesis.
\end{proof}

Since, for $\ell\ne k$,
$$u_i(x^{\ell}_i)u_i \to h_{\ell}=\sum_{j=1}^Na_{\ell}^jG_{\overline{x}^j}$$
by comparison with equation \eqref{LimiteFunzioneGreen}, one gets that
\begin{equation}\label{CoefficienteBlowUp}
 a_k^j =
 \frac{32\pi}{K(\overline{x}^j)}
 \lim_{i\to\infty}\frac{u_i(x_i^k)}{u_i(x_i^j)}.
\end{equation}

From now on, up to passing to a subsequence we will suppose that the limit
\begin{equation}\label{LimiteTau}
 \mi_k = \lim_{i\to\infty}\tau_iu_i(x_i^k)^2\in[0,\infty)
\end{equation}
exists for every $k$ (the limit is finite because of Proposition \ref{StimeCitazione}).

\begin{proposizione}\label{PropMatriceDefinita}
 If $u_i$ is a sequence of solutions of equation \eqref{EquazioneBlowUp} with $p_i\to 3$ with blow up set $S=\{\overline{x}^1,\ldots,\overline{x}^N\}$ then in CR normal coordinates
 $$-\frac{\Delta_bK(\overline{x}^k)}{K(\overline{x}^k)} -32A_{\overline{x}^k}\ge 0$$
 and if $N\ge 2$, the strict inequality holds.
 Furthermore the matrix $M(S)$ defined in formula \eqref{DefinizioneMatrice} is positive semidefinite.
\end{proposizione}

\begin{proof}
 Let us apply the Pohozaev identity (see Proposition 3.3 in \cite{Af}) to $\widetilde{u}_i\circ\delta_{\sigma}$ (where $\widetilde{u}_i$ has been defined in the proof of Lemma \ref{SegnoTermineOrdineZero}):
 $$\int_{B_1(\overline{x}^k)}\left(\left(\frac{1}{p_i+1}-\frac{1}{4}\right)(\phi\circ\delta_{\sigma})^{-\tau_i}K\circ\delta_{\sigma}(\widetilde{u}_i\circ\delta_{\sigma})^{p_i+1} -\frac{1}{4}(R\circ\delta_{\sigma})(\widetilde{u}_i\circ\delta_{\sigma})^2+ \right.$$
 $$+\frac{1}{4}\frac{1}{p_i+1}\Xi((\phi\circ\delta_{\sigma})^{-\tau_i}K\circ\delta_{\sigma})(\widetilde{u}_i\circ\delta_{\sigma})^{p_i+1}-\frac{1}{8}\Xi(R\circ\delta_{\sigma})(\widetilde{u}_i\circ\delta_{\sigma})^2+$$
 $$\left.-\left(\Xi (\widetilde{u}_i\circ\delta_{\sigma})+(\widetilde{u}_i\circ\delta_{\sigma})\right)(\Delta_b(\widetilde{u}_i\circ\delta_{\sigma})-\cerchio{\Delta}_b(\widetilde{u}_i\circ\delta_{\sigma}))\right)d\cerchio{V}=$$
 $$=\int_{\de B_1(\overline{x}^k)}\left(\left(\frac{1}{4}\frac{1}{p_i+1}(\phi\circ\delta_{\sigma})^{-\tau_i}K\circ\delta_{\sigma}(\widetilde{u}_i\circ\delta_{\sigma})^{p_i+1}-\frac{1}{8}(R\circ\delta_{\sigma})(\widetilde{u}_i\circ\delta_{\sigma})^2  \right)\Xi\cdot\cerchio{\ni} +\right.$$
 $$\left.+\ci{B}(x,\widetilde{u}_i\circ\delta_{\sigma},\cerchio{\nabla}^H(\widetilde{u}_i\circ\delta_{\sigma}))\right)d\cerchio{\sigma}.$$
 As shown in the proofs of Lemma 4.14 in \cite{Af} and of Lemma \ref{SegnoTermineOrdineZero} above, applying $\lim_{\sigma\to 0}\sigma^3\limsup_{i\to\infty}u_i(x_i^k)^2$ to the above all terms tend to zero, except the one involving $\ci{B}$, that by Proposition 3.4 in \cite{Af} and equation \eqref{LimiteFunzioneGreen} tends to
 $$\frac{32\pi}{K(\overline{x}^k)}\left(-\frac{32\pi}{K(\overline{x}^k)}A_{\overline{x}^k} -\sum_{j\ne k}a_k^jG_{\overline{x}^j}(\overline{x}^k)\right),$$
 and the first and third term in the right hand side.
 The first one is equal to
 $$\left(\frac{1}{p_i+1}-\frac{1}{4}\right)\int_{B_1(\overline{x}^k)}(\phi\circ\delta_{\sigma})^{-\tau_i}K\circ\delta_{\sigma}(\widetilde{u}_i\circ\delta_{\sigma})^{p_i+1} = $$
 $$= \frac{\tau_i}{4(p_i+1)}u_i(x_i^k)^{-2(p_i-1)}\sigma^{-4}\int_{B_{\sigma R_i^k}(\overline{x}^k)}\left(\phi\circ\delta_{(M_i^k)^{-\frac{p_i-1}{2}}}\right)^{-\tau_i}K\circ\delta_{(M_i^k)^{-\frac{p_i-1}{2}}}\left(\widetilde{u}_i\circ\delta_{(M_i^k)^{-\frac{p_i-1}{2}}}\right)^{p_i+1} +$$
 $$+\frac{\tau_i}{4(p_i+1)}\int_{B_1(\overline{x}^k)\setminus B_{r_i}(\overline{x}^k)}O\left(u_i(x_i^k)^{-(p_i-1)}|x|^{-2(p_i-1)}\right)$$
 where $M_i^k=u_i(x_i^k)$. Using equation \eqref{LimiteTau} and standard estimates, when multiplying by $u_i(x_i^k)^2$ and taking the limit for $i\to\infty$ the last term tends to zero, the first one to
 $$\frac{1}{16}\mi_kK(\overline{x}^k)\int_{\H^1}\left(U\circ\delta_{K(\overline{x}^k)^{1/2}/2}\right)^{4} = \frac{1}{16}\mi_kK(\overline{x}^k)16K(\overline{x}^k)^{-2}\int_{\H^1}U^{4}=
 \frac{\pi^2}{4}\mi_k\frac{1}{K(\overline{x}^k)}.$$
 
 The other term to estimate is
 $$\frac{1}{4(p_i+1)}\int_{B_1(\overline{x}^k)}\Xi((\phi\circ\delta_{\sigma})^{-\tau_i}K\circ\delta_{\sigma})(\widetilde{u}_i\circ\delta_{\sigma})^{p_i+1}=$$
 $$= \frac{1}{4(p_i+1)}\int_{B_1(\overline{x}^k)}\left(\Xi((\phi\circ\delta_{\sigma})^{-\tau_i})K\circ\delta_{\sigma} + (\phi\circ\delta_{\sigma})^{-\tau_i}\Xi(K\circ\delta_{\sigma})\right)(\widetilde{u}_i\circ\delta_{\sigma})^{p_i+1}.$$
 It is easy to verify that when multiplying by $u_i(x_i^k)^2$ and taking the limit for $i\to\infty$ the first term tends to zero, so we have to estimate
 $$\frac{1}{4(p_i+1)}\int_{B_1(\overline{x}^k)}(\phi\circ\delta_{\sigma})^{-\tau_i}\Xi(K\circ\delta_{\sigma})(\widetilde{u}_i\circ\delta_{\sigma})^{p_i+1} =$$
 $$= \frac{1}{4(p_i+1)}\int_{B_1(\overline{x}^k)}(\phi\circ\delta_{\sigma})^{-\tau_i}(x\cerchio{X} +y\cerchio{Y}+2t\cerchio{T})\left[K(\overline{x}^k) + x(XK)(\overline{x}^k)+\right.$$
 $$+ y(YK)(\overline{x}^k) +xy(XYK)(\overline{x}^k) + \frac{1}{2}x^2(X^2K)(\overline{x}^k) + \frac{1}{2}y^2(Y^2K)(\overline{x}^k) + t(TK)(\overline{x}^k)+$$
 $$\left.+O(|x|^3)\right]\circ\delta_{\sigma}(\widetilde{u}_i\circ\delta_{\sigma})^{p_i+1} =$$
 $$= \frac{1}{4(p_i+1)}\int_{B_1(\overline{x}^k)}(\phi\circ\delta_{\sigma})^{-\tau_i}(x(XK)(\overline{x}^k) + y(YK)(\overline{x}^k) +$$
 $$+2xy(XYK)(\overline{x}^k) + x^2(X^2K)(\overline{x}^k) +y^2(Y^2K)(\overline{x}^k) + 2t(TK)(\overline{x}^k) +O(|x|^3))\circ\delta_{\sigma}(\widetilde{u}_i\circ\delta_{\sigma})^{p_i+1}=$$
 $$= O\left(\int_{B_1(\overline{x}^k)}u_i(\overline{x}^k)^{-1}x\circ\delta_{\sigma}(u_i\circ\delta_{\sigma})^{p_i+1}\right)+$$
 $$+O\left(\int_{B_1(\overline{x}^k)}(\phi\circ\delta_{\sigma})^{-\tau_i}(xy)\circ\delta_{\sigma}(u_i\circ\delta_{\sigma})^{p_i+1}\right)+$$
 $$+\frac{1}{4(p_i+1)}4\Delta_bK(x_i^k)\int_{B_1(\overline{x}^k)}(\phi\circ\delta_{\sigma})^{-\tau_i}x^2\circ\delta_{\sigma}(\widetilde{u}_i\circ\delta_{\sigma})^{p_i+1}+$$
 $$+O\left(\int_{B_1(\overline{x}^k)}t\circ\delta_{\sigma}u_i(\overline{x}^k)^{-1}(\widetilde{u}_i\circ\delta_{\sigma})^{p_i+1}\right)+ O\left(\int_{B_1(\overline{x}^k)}|x|^3\circ\delta_{\sigma}(\widetilde{u}_i\circ\delta_{\sigma})^{p_i+1}\right)$$
 Standard computations show
 that applying $\lim_{\sigma\to 0}\sigma^3\limsup_{i\to\infty}u_i(x_i^k)^2$ the result is
 $$\frac{1}{4}\Delta_bK(\overline{x}^k)\int_{\H^1}x^2(U\circ\delta_{K(\overline{x}^k)^{1/2}/2})^4 = \frac{1}{4}\Delta_bK(\overline{x}^k)16K(\overline{x}^k)^{-2}\int_{\H^1}4K(\overline{x}^k)^{-1}|x|^2U^4=$$
 $$= \frac{16\Delta_bK(\overline{x}^k)}{K(\overline{x}^k)^3}\int_{\H^1}x^2U^4 = \frac{32
 \pi^2\Delta_bK(\overline{x}^k)}{K(\overline{x}^k)^3}.$$
 Thus putting all together we get that
 $$\frac{\pi^2}{4}\mi_k\frac{1}{K(\overline{x}^k)} +\frac{32\pi^2\Delta_bK(\overline{x}^k)}{K(\overline{x}^k)^3} =
 \frac{32\pi}{K(\overline{x}^k)}\left(-\frac{32\pi}{K(\overline{x}^k)}A_{\overline{x}^k} -\sum_{j\ne k}a_k^jG_{\overline{x}^j}(\overline{x}^k)\right)$$
 and since $a_k^j>0$ and $\mi_k\ge 0$ the first part of the thesis is proved.
 
 Now let us define $\lambda_j=\frac{1}{K(\overline{x}^j)^{1/2}}\lim_{i\to\infty}\frac{u_i(\overline{x}^1)}{u_i(\overline{x}^j)}$. Then by formula \eqref{CoefficienteBlowUp} $a_k^j = \frac{32\pi}{K(\overline{x}^k)^{1/2}K(\overline{x}^j)^{1/2}}\frac{\lambda_j}{\lambda_k}$, and therefore the last formula can be written as
 $$
 \frac{\mi_j}{K(\overline{x}^j)} +C_2\frac{\Delta_bK(\overline{x}^j)}{K(\overline{x}^j)^2} = -A_{\overline{x}^j} -\sum_{k\ne j}\frac{C}{K(\overline{x}^k)^{1/2}K(\overline{x}^j)^{1/2}}\frac{\lambda_k}{\lambda_j}G_{\overline{x}^k}(\overline{x}^j)
 $$
 $$\frac{\pi}{128}\mi_k +\frac{\pi\Delta_bK(\overline{x}^k)}{K(\overline{x}^k)^2} =
 -\frac{32\pi}{K(\overline{x}^k)}A_{\overline{x}^k} -\sum_{j\ne k}\frac{32\pi}{K(\overline{x}^k)^{1/2}K(\overline{x}^j)^{1/2}}\frac{\lambda_j}{\lambda_k}G_{\overline{x}^j}(\overline{x}^k)$$
 that is, recalling formula \eqref{DefinizioneMatrice},
 \begin{equation}\label{EquazioneMatrice}
  \sum_{j=1}^NM_{jk}\lambda_j = \frac{1}{128}\mi_k\lambda_k.
 \end{equation}
 By Lemma \ref{LemmaAlgebraLineare} there exists a vector $v=(v_1,\ldots,v_N)\ne 0$ such that $M(S)v= \mi(M(S))v$
 and $v_k\ge 0$ for every $k$.
 Multiplying equation \eqref{EquazioneMatrice} by $v_k$ and summing over $k$ one gets
 $$\mi(M(S))\sum_{j=1}^N\lambda_jv_j = \sum_{j=1}^N\sum_{k=1}^NM_{jk}\lambda_j = \sum_{k=1}^N\mi_k\lambda_kv_k \ge 0$$
 and since $\lambda_j>0$ for every $j$ and $v\ne 0$, we proved the second part of the thesis.
\end{proof}

In the above proof we needed a lemma in linear algebra which taken from \cite{BCCH}, which we breafly recall.

\begin{lemma}\label{LemmaAlgebraLineare}
 Let $M$ be a $N\times N$ symmetric matrix suh that $M_{ij}<0$ for $i\ne j$. Then there exist an eigenvector $v$ for the least eigenvalue $\mi(M)$ such that $v_i\ge 0$ for every $i$.
\end{lemma}

\begin{proof}
 Let $v_i$ an eigenvector for $\mi(M)$ with $\N{v}=1$. Then
 $$Mv \cdot v = \sum_{i=1}^NM_{ii}v_i^2 + \sum_{i\ne j}M_{ij}v_iv_j \ge \sum_{i=1}^NM_{ii}|v_i|^2 + \sum_{i\ne j}M_{ij}|v_i||v_j| = Mw \cdot w$$
 where $w=(|v_1|,\ldots,|v_N|)$. Since $\N{w}=1$ this implies that $w$ is an eigenvector for $\mi(M)$.
\end{proof}

\begin{proof}[Proof of Theorem \ref{Teorema1}]
 By regularity theory (see section 2.2 in \cite{Af} and the references cited therein)
 it is enough to prove that the set of solutions is bounded in $L^{\infty}$.
 Were this not the case, there would exist a sequence of solutions blowing up at some non empty set $S=\{\overline{x}^1,\ldots,\overline{x}^N\}$.
 
 Since in this case $\tau_i=0$ for every $i$, then $\mi_i=0$ (where $\mi_i$ is defined in \eqref{LimiteTau}).
 So Equation \eqref{EquazioneMatrice} becomes $M(S)\lambda=0$ with $\lambda=(\lambda_1,\ldots,\lambda_N)$.
 By Lemma \ref{LemmaAlgebraLineare} there exists $v\ne 0$ with nonnegative components and $M(S)v=\mi(M(S))v$.
 
 If $\mi(M(S))\ne 0$, then $\lambda$ and $v$ are two eigenvectors of the symmetric matrix $M(S)$ with respect to distinct eigenvalues, and so they are orthogonal. But this is impossible, since they are non zero vectors with nonnegative components. So $\mi(M(S))=0$. But this is against the hypothesis.
\end{proof}

\section{Degree for subcritical approximation}\label{SezioneGrado}
Let
$$S_p(u)= u - L_{\theta}^{-1}(u^p)$$
for $p\in[1,3]$, $S_p:C^{2,\alfa}(M)\to C^{2,\alfa}(M)$.

\begin{lemma}\label{LimitatezzaSoluzioniSottocritiche}
 Given $K>0$, for every $\e>0$ there exists $C>0$ depending only on $M$, $\N{K}_{C^2}$ and $\inf K$ such that every solution of
 $L_{\theta}u = Ku^p$ with $u>0$ and $p\in[1+\e,3-\e]$ satisfies $u>\frac{1}{C}$ and $\N{u}_{C^{2,\alfa}}\le C$.
\end{lemma}

\begin{proof}
 The result is proved by a standard blow-up argument (see for example the proof of Proposition 4.5 in \cite{Af}) and finding a contradiction with Ma-Ou's non existence theorem (Theorem 1.1 in \cite{MO}).
\end{proof}

Let $C$ be such that every solution of $S_p(u)=0$ belongs to $\Omega_C$ where
$$\Omega_C = \left\{ u\in \Gamma^{2,\alfa}(M) \;|\; \N{u}_{\Gamma^{2,\alfa}}<C, u>\frac{1}{C}\right\}.$$
We need to compute $\deg(F_p,\Omega_C,0)$.

Let
$$\widetilde{S}_p(u)= u-L_{\theta}^{-1}(f(u)u^p),$$
where $f(u)=\int uL_{\theta}u$, and
$$\widetilde{S}_{p,t}(u)= u-L_{\theta}^{-1}(f_t(u)u^p)$$
where $f_t(u)=1-t+tf(u)$.
By the properties of the Leray-Schauder degree, the degree does not depend by $C$ if it is big enough; therefore let us choose $C$ such that $u\in\Omega_C$ for every $u$ such that $\widetilde{S}_{p,t}(u)=0$ for some $t\in[0,1]$ (which exists by Lemma \ref{LimitatezzaSoluzioniSottocritiche}). This allows to build an admissible homotopy, and so $\deg(F_p,\Omega_C,0)=\deg(\widetilde{S}_p,\Omega_C,0)$.

Notice that the set $\{u\in \Gamma^{2,\alfa}(M) \;|\; u>0, \widetilde{S}_1(u) = 0\}$ consists of the eigenfunction $\psi_1$ for the first eigenvalue $\lambda_1$ of $L_{\theta}$ normalized in such a way that $f(\psi_1)=\lambda_1$.

\begin{lemma}\label{ZeriSulBordo}
 If $C$ is big enough then $\de\Omega_C$ does not contain zeroes of $\widetilde{S}_p$ for $p\in[1,3-\e]$.
\end{lemma}

\begin{proof}
 It is sufficient to prove that there does not exist a sequence $u_n$ with $\widetilde{S}_{p_n}(u_n)=0$, $u_n\in\de\Omega_C$ and $p_n\to 1$, because if such a sequence does not exist, the thesis can be obtained from Lemma \ref{LimitatezzaSoluzioniSottocritiche} up to choosing $C$ big enough.
 
 $u_n\in\de\Omega_C$ if and only if either $\N{u_n}_{\Gamma^{2,\alfa}}=C$ or $\min u_n=\frac{1}{C}$. Since functions in $\overline{\Omega_C}$ are bounded in $\Gamma^{2,\alfa}$, by regularity theory $u_n\to\overline{u}$ in $\Gamma^{2,\alfa}$, with $\overline{u}\in\de\Omega_C$ satisfying $L_{\theta}\overline{u}=f(\overline{u})\overline{u}$, that is $\widetilde{S}_1(\overline{u})=0$. But the only zero of $\widetilde{S}_1$ in $\overline{\Omega_C}$ is $\psi_1$,
 and if $C$ is chosen big enough then $\psi_1\notin\de\Omega_C$, so
 there is a contradiction.
\end{proof}

\begin{lemma}\label{GradoLineare}
 $\deg(\widetilde{S}_1,\Omega_C,0)=-1$.
\end{lemma}

\begin{proof}
 The differential of $\widetilde{S}_1$ at $\psi_1$ is
 $$d\widetilde{S}_1(\psi_1)[u] = u -df(\psi_1)[u]L_{\theta}^{-1}\psi_1 - f(\psi_1)L_{\theta}^{-1}u =$$
 \begin{equation}\label{DifferenzialeFunzionale}
  =u - 2\left(\int vL_{\theta}\psi_1\right)L_{\theta}^{-1}\psi_1 - \lambda_1L_{\theta}^{-1}u = u - 2\left(\int u\psi_1\right)\psi_1 - \lambda_1L_{\theta}^{-1}u.
 \end{equation}
 Notice that
 $$d\widetilde{S}_1(\psi_1)[\psi_1]=\psi_1 - 2\left(\int\psi_1\lambda_1^{-1}L_{\theta}\psi_1\right)\psi_1 - \lambda_1L_{\theta}^{-1}\psi_1=$$
 $$= \psi_1 - 2\lambda_1^{-1}\lambda_1\psi_1 -\psi_1 = -2\psi_1$$
 therefore $\psi_1$ is an eigenvector of $d\widetilde{S}_1(\psi_1)$ with eigenvalue $-2$.
 $d\widetilde{S}_1(\psi_1)$ is self-adjoint in $L^2$, because
 $$\bra d\widetilde{S}_1(\psi_1)[u],v\ket_{L^2} = \int uv - 2\left(\int u\psi_1\right)\left(\int v\psi_1\right) -\lambda_1L_{\theta}^{-1}\int uv = \bra d\widetilde{S}_1(\psi_1)[v],u\ket_{L^2}.$$
 Therefore equation \eqref{DifferenzialeFunzionale} implies that eigenvectors of $d\widetilde{S}_1(\psi_1)$ coincide with eigenvectors of $L_{\theta}$, and that if $\psi$ is an eigenvector of $L_{\theta}$ with eigenvalue $\lambda\ne\lambda_1$, then $\psi$ is an eigenvector of $d\widetilde{S}_1(\psi_1)$ with eigenvalue $1-\frac{\lambda_1}{\lambda}$. Since $\lambda_1$ is the least eigenvalue of $L_{\theta}$, $-2$ is the only negative eigenvalue of $d\widetilde{S}_1(\psi_1)$, and therefore the thesis follows from the properties of the Leray-Schauder degree.
\end{proof}

\begin{lemma}
 If $p\in(1,3)$ then $\deg(S_p,\Omega_C,0)=-1$.
\end{lemma}

\begin{proof}
 It follows from Lemmas \ref{ZeriSulBordo} and \ref{GradoLineare} and from the invariance by admissible homotopy of the Leray-Schauder degree.
\end{proof}

Now given $K>0$, $K\in C^2(M)$ let us define
$$F_p(u)= u -L_{\theta}^{-1}(Ku^p)$$
and
$$F_{p,t}(u)= u -L_{\theta}^{-1}(K_tu^p)$$
where $K_t= (1-t) +tK$, for $p\in[1,3]$, $F_{p,t}:\Gamma^{2,\alfa}(M)\to \Gamma^{2,\alfa}(M)$.
Therefore $F_{p,0}=S_p$ and $F_{p,1}=F_p$.

By Lemma \ref{LimitatezzaSoluzioniSottocritiche}, given $p\in(1,3)$, up to choosing $C$ big enough, $\bigcup_t F_{p,t}^{-1}(\{0\})\subset\Omega_C$. Therefore, by homotopy invariance, we can compute the degree of $F_p$.

\begin{proposizione}\label{ProposizioneGradoSottocritico}
 If $p\in(1,3)$ then $\deg(F_p,\Omega_C,0)=-1$.
\end{proposizione}

With the hypothesis of positivity of the pseudohermitian mass of blow-ups at every point, thanks to Theorem 1.1 in \cite{Af} we can prove the following result, which we will not need but is of independent interest.

\begin{teorema}\label{IlTeorema}
 If for every $x\in M$, the pseudohermitian mass of $M\setminus\{x\}$ with its blow-up structure as defined in \cite{CMY1} is positive, then $\deg(S_3,\Omega_C,0)=-1$.
\end{teorema}

\section{Degree for points at infinity}\label{SezioneSoluzioniApprossimate}
Let $K>0$, $K\in C^2(M)$ be a Morse function, and let
$$\ci{J}_{\tau}(u) =  \frac{1}{2}\int uL_{\theta}u - \frac{1}{4-\tau}\int K|u|^{4-\tau}.$$
Given a set $S=\{\overline{x}^1,\ldots,\overline{x}^N\}$ of critical points of $K$ such that $M(S)$ is positive definite, we want to build approximate critical points of $\ci{J}_{\tau}$ blowing up at $S$.

Since the results of this section are invariant by changes of contact form, throughout this section we will assume that $\theta$ is in CR normal coordinates around any critical point of $K$.

Given $p\in M$ in some neighborhood of a critical point $\overline{x}$ of $K$, let $\Phi_p:B_{\rho}(p)\to\H^1$ the map defining pseudohermitian normal coordinates around $p$ (with $\rho$ independent from $p$ by compactness), and let $\chi$ be a cut-off function equal to $1$ in $B_{\rho/2}(\overline{x})$ and to zero in $M\setminus B_{\rho}(\overline{x})$.
Then we define
$$\phi_{p,\lambda}= (U_{\lambda}\circ\Phi_p)\chi$$
where $U_{\lambda}=\lambda U\circ\delta_{\lambda}$, $U$ being defined in formula \eqref{BollaFormula}.

$\phi_{p,\lambda}$ is a family of approximate solutions, but it is not appropriate to our aims because, consisting of functions compactly supported in a neighborhood of a critical point, it fails to take in consideration the interactions among different bubble, which in three dimensional manifolds is not negligible.
So we define $\widetilde{\phi}_{p,\lambda}$ as the solution of
$$-L_{\theta}\widetilde{\phi}_{p,\lambda} = \phi_{p,\lambda}^3,$$
that is, explicitely,
$$\widetilde{\phi}_{p,\lambda}(x) = \int_M G(x,y)\phi_{p,\lambda}^3(y)dy.$$

\begin{lemma}\label{StimaDifferenza}
 $$\N{\widetilde{\phi}_{p,\lambda}-\phi_{p,\lambda}}_{H^1}\lesssim\frac{1}{\lambda}, \;\;\; \N{\widetilde{\phi}_{p,\lambda}-\phi_{p,\lambda}}_{L^{\infty}}\lesssim\frac{1}{\lambda},$$
 $$\N{\frac{\de\widetilde{\phi}_{p,\lambda}}{\de\lambda}-\frac{\de\phi_{p,\lambda}}{\de\lambda}}_{L^{\infty}}\lesssim\frac{1}{\lambda^2}, \;\;\; \N{\frac{\de\widetilde{\phi}_{p,\lambda}}{\de p}-\frac{\de\phi_{p,\lambda}}{\de p}}_{L^{\infty}}\lesssim 1$$
\end{lemma}

\begin{proof}
 It follows from the definitions that
 $$L_{\theta}(\widetilde{\phi}_{p,\lambda}-\phi_{p,\lambda})\lesssim\frac{1}{\lambda}, \;\;\; L_{\theta}\left(\frac{\de\widetilde{\phi}_{p,\lambda}}{\de\lambda}-\frac{\de\phi_{p,\lambda}}{\de\lambda}\right)\lesssim\frac{1}{\lambda^2},\;\;\; L_{\theta}\left(\frac{\de\widetilde{\phi}_{p,\lambda}}{\de p}-\frac{\de\phi_{p,\lambda}}{\de p}\right)\lesssim 1$$
 so the thesis follows from the maximum principle.
\end{proof}

\begin{lemma}\label{LimiteBolle}
 $$\lambda(\widetilde{\phi}_{p,\lambda}-\phi_{p,\lambda})\to G_p$$
 uniformly on compacts in $M\setminus\{p\}$, and
 $$\lambda(\widetilde{\phi}_{p,\lambda}(p)-\phi_{p,\lambda}(p))\to A_p.$$
\end{lemma}

\begin{proof}
 It is Lemma 2.3 in \cite{G}.
\end{proof}

Given $\bm{p}=(p^1,\ldots,p^N)\in M^N$, $\bm{a}=(a_1,\ldots,a_N)\in\R^N$ and $\bm{\lambda}=(\lambda_1,\ldots,\lambda_N)\in\R^N$, define
$$\Psi_{\bm{p},\bm{\lambda},\bm{a}} = \sum_{k=1}^Na_k\phi_{p_k,\lambda_k}.$$
and
$$\widetilde{\Psi}_{\bm{p},\bm{\lambda},\bm{a}} = \sum_{k=1}^Na_k\widetilde{\phi}_{p_k,\lambda_k}.$$

From Lemma \ref{StimaDifferenza} we can deduce the following estimates.

\begin{lemma}\label{StimaDifferenzaTotale}
 $$\N{\widetilde{\Psi}_{\bm{p},\bm{\lambda},\bm{a}}-\Psi_{\bm{p},\bm{\lambda},\bm{a}}}_{H^1}\lesssim\frac{1}{\lambda}, \;\;\; \N{\widetilde{\Psi}_{\bm{p},\bm{\lambda},\bm{a}}-\Psi_{\bm{p},\bm{\lambda},\bm{a}}}_{L^{\infty}}\lesssim\frac{1}{\lambda},$$
 $$\N{\frac{\de\widetilde{\Psi}_{\bm{p},\bm{\lambda},\bm{a}}}{\de\lambda}-\frac{\de\Psi_{\bm{p},\bm{\lambda},\bm{a}}}{\de\lambda}}_{L^{\infty}}\lesssim\frac{1}{\lambda^2}, \;\;\; \N{\frac{\de\widetilde{\Psi}_{\bm{p},\bm{\lambda},\bm{a}}}{\de p}-\frac{\de\Psi_{\bm{p},\bm{\lambda},\bm{a}}}{\de p}}_{L^{\infty}}\lesssim 1$$
\end{lemma}

Given $S=\{\overline{x}^1,\ldots,\overline{x}^N\}$ and constants $A>0$, $\e>0$, define
$$\Sigma_{\tau}(S) = \left\{\widetilde{\Psi}_{\bm{p},\bm{\lambda},\bm{a}}\;\middle|\;d(p_k,\overline{x}^k)<A\tau^{\frac{1}{4}},A^{-1}\tau^{-\frac{1}{2}}<\lambda_i<A\tau^{-\frac{1}{2}},\left|a_k-\frac{2}{K(\overline{x}^k)^{\frac{1}{2}}}\right|<\e\right\}$$
and
$$U_{\tau}(S) = \left\{\widetilde{\Psi}_{\bm{p},\bm{\lambda},\bm{a}}+v\;\middle|\; \widetilde{\Psi}_{\bm{p},\bm{\lambda},\bm{a}}\in\Sigma_{\tau}(S), \N{v}_{H^1}<\e \right\}.$$

\begin{proposizione}\label{DicotomiaSoluzioni}
 If the nondegeneracy conditions \eqref{Condizione1} and \eqref{Condizione2} hold, there exist $\tau_0>0$, $R>0$, such that choosing $A$ big enough and $\e$ small enough in the definition of $\Sigma_{\tau}(S)$, it holds that if $u>0$ satisfies $d\ci{J}_{\tau}(u)=0$ then
 $$u\in \Omega_R\cup\bigcup_{S\subseteq\operatorname{Crit}(K):M(S)>0}U_{\tau}(S)$$
 where $\Omega_R$ is defined in Section \ref{SezioneGrado}.
\end{proposizione}

\begin{proof}
 Let $\tau_i\to 0$ and $u_i>0$ such that $d\ci{J}_{\tau_i}(u_i)=0$ and $u_i\not\in\Omega_R$. By Harnack inequality $\N{u_i}_{C^{2,\alfa}}\to\infty$. We have to prove that up to subsequences $u_i$ belongs eventually $U_{\tau}(S)$ for some $S\subseteq\Crit(K)$ such that $M(S)>0$.
 
 Let $S$ the set of blow-up points. By Theorem \ref{TeoremaBlowUpIsolati} and Proposition \ref{PropMatriceDefinita} $S=\{\overline{x}^1,\ldots,\overline{x}^N\}$ is a set of critical points of $K$ such that $M(S)\ge 0$, and thus, by hypothesis \eqref{Condizione2}, that $M(S)>0$; furthermore the blow-up points are isolated simple.
 Let $x_i^1,\ldots,x_i^N$ be the sequences of local maxima coming from the definition of isolated blow-up point.
 
 By Proposition \ref{BlowUpIsolatiBolle}, calling $p_i=3-\tau_i$, $\lambda_i^k=\frac{K(x_i^k)^{\frac{1}{2}}}{2}u_i(x_i^k)^{\frac{2}{p_i-1}}$ and $a_i^k=\left(\frac{2}{K(x_i^k)^{\frac{1}{2}}}\right)^{\frac{2}{p_i-1}}(\lambda_i^k)^{\frac{\tau_i}{p_i-1}}$, we have
 $$\N{u_i-a_i^k\phi_{x_i^k,\lambda_i^k}}_{H^1(B_{r_i}(x_i^k))}\to 0.$$
 and $\frac{1}{A}\le\tau_iu_i(x_i^k)^2\le A$ and $\left|a_i^k-\frac{2}{K(x_i^k)^{\frac{1}{2}}}\right|<\e$ for some $A,\e$.
 Furthermore, since $K$ is a Morse function, $\frac{1}{C}d_{g_{\theta}}(x,\overline{x}^k)\le|\nabla_{g_{\theta}}K(x)|\le Cd_{g_{\theta}}(x,\overline{x}^k)$ for some $C$ in some neighborhood of $\overline{x}^k$ for every $k$, and so, by Proposition \ref{GradienteNullo},
 $$d(x_i^k,\overline{x}^k)\lesssim d_{g_{\theta}}(x_i^k,\overline{x}^k)^{\frac{1}{2}}\lesssim|\nabla_{g_{\theta}}K(x_i^k)|^{\frac{1}{2}}\lesssim u_i(x_i^k)^{-\frac{1}{2}}.$$
 
 Therefore if we prove that $\N{u_i-\Psi_{\bm{p}_i,\bm{\lambda}_i,\bm{a}_i}}_{H^1}\le\e$, where $\bm{p}_i=(p_i^1,\ldots,p_i^N)$, $\bm{\lambda}_i=(\lambda_i^1,\ldots,\lambda_i^N)$ and $\bm{a}_i=(a_i^1,\ldots,a_i^N)$, then by Lemma \ref{StimaDifferenza} $u_i\in U_{\tau}(S)$.
 
 Since $\overline{x}^k$ is an isolated blow-up point,
 $$\int_{B_{\rho}(x_i^k)\setminus B_{r_i}(x_i^k)}|\nabla_bu_i|^2+|u_i|^2\lesssim\int_{B_{\rho}(x_i^k)\setminus B_{r_i}(x_i^k)}u_i(x_i^k)^{-2}|x|^{-6}\lesssim u_i(x_i^k)^{-2}\left(R_i^ku_i(x_i^k)^{-\frac{p_i-1}{2}}\right)^{-2}\to 0$$
 (where $R_i^k$ is the distance defined in Proposition \ref{BlowUpIsolatiBolle})
 therefore $\N{u_i}_{H^1(B_{\rho}(x_i^k)\setminus B_{r_i}(x_i^k))}\to 0$.
 Furthermore
 $$\int_{B_{\rho}(x_i^k)\setminus B_{r_i}(x_i^k)}|\nabla_b\Psi_{\bm{p}_i,\bm{\lambda}_i,\bm{a}_i}|^2+|\Psi_{\bm{p}_i,\bm{\lambda}_i,\bm{a}_i}|^2= \int_{B_{\rho}(x_i^k)\setminus B_{r_i}(x_i^k)}|\nabla_b(a_i^k\phi_{p_i^k,\lambda_i^k})|^2+|a_i^k\phi_{p_i^k,\lambda_i^k}|^2\lesssim$$
 $$\lesssim\int_{B_{\rho}(x_i^k)\setminus B_{r_i}(x_i^k)}(\lambda_i^k)^2\left|\lambda_i^k(1+\lambda_i^k|x|)^{-3}\right|^2 \lesssim \int_{B_{\rho}(x_i^k)\setminus B_{r_i}(x_i^k)}(\lambda_i^k)^4(\lambda_i^k)^{-6}|x|^{-6}\lesssim$$
 $$\lesssim(\lambda_i^k)^{-2}\left(R_i^ku_i(x_i^k)^{-\frac{p_i-1}{2}}\right)^{-2}\to 0$$
 therefore $\N{\Psi_{\bm{p}_i,\bm{\lambda}_i,\bm{a}_i}}_{H^1(B_{\rho}(x_i^k)\setminus B_{r_i}(x_i^k))}\to 0$. Thus $\N{u_i-\Psi_{\bm{p}_i,\bm{\lambda}_i,\bm{a}_i}}_{H^1(B_{\rho}(x_i^k)\setminus B_{r_i}(x_i^k))}\to 0$.
 
 Analogously it can be proved that $\N{u_i-\Psi_{\bm{p}_i,\bm{\lambda}_i,\bm{a}_i}}_{H^1(M\setminus\bigcup_{k=1}^N B_{\rho}(\overline{x}^k))}\to 0$.
\end{proof}

The following Proposition is classical in conformal geometry; we state it in our case without repeating the proof, which is the same.

\begin{proposizione}
 Given $A,\e,$ there exist $A'$ big enough and $\e',\tau_0>0$ small enough such that if $0<\tau<\tau_0$ for every $u\in U_{\tau}(S)$ the minimization problem
 $$\min\left\{\N{u-\widetilde{\Psi}_{\bm{p},\bm{\lambda},\bm{a}}}_{H^1}\;\middle|\;d(p_k,\overline{x}^k)<A'\tau^{\frac{1}{4}},A'^{-1}\tau^{-\frac{1}{2}}<\lambda_i<A'\tau^{-\frac{1}{2}},\left|a_k-\frac{2}{K(\overline{x}^k)^{\frac{1}{2}}}\right|<\e'\right\}$$
 has unique solution, and if $v=u-\widetilde{\Psi}_{\bm{p},\bm{\lambda},\bm{a}}$ then
 $$\bra v, \widetilde{\phi}_{p_k,\lambda_k}\ket = \bra v, \frac{\de\widetilde{\Psi}_{\bm{p},\bm{\lambda},\bm{a}}}{\de p^k}\ket = \bra v, \frac{\de\widetilde{\Psi}_{\bm{p},\bm{\lambda},\bm{a}}}{\de\lambda^k}\ket =0.$$
\end{proposizione}

As for the above proposition, notice that 
$$\widetilde{\phi}_{p_k,\lambda_k}=\frac{\de\widetilde{\Psi}_{\bm{p},\bm{\lambda},\bm{a}}}{\de a^k}.$$

The proposition implies that if we define
$$H_{\bm{p},\bm{\lambda},\bm{a}} = \bra\widetilde{\phi}_{p_k,\lambda_k},\frac{\de\widetilde{\Psi}_{\bm{p},\bm{\lambda},\bm{a}}}{\de p^k},\frac{\de\widetilde{\Psi}_{\bm{p},\bm{\lambda},\bm{a}}}{\de\lambda^k}\ket^{\perp}$$
and
$$V_{\tau}(S)= \left\{\widetilde{\Psi}_{\bm{p},\bm{\lambda},\bm{a}}+v\;\middle|\; \widetilde{\Psi}_{\bm{p},\bm{\lambda},\bm{a}}\in\Sigma_{\tau}(S), \N{v}_{H^1}<\e,v\in H_{\bm{p},\bm{\lambda},\bm{a}}\right\}$$
then Proposition \ref{DicotomiaSoluzioni} holds with $V_{\tau}(S)$ in place of $U_{\tau}(S)$.

\begin{proposizione}\label{DifferenzialeSecondoPositivo}
 There exists $C$ such that for $0<\tau<\tau_0$, $\widetilde{\Psi}_{\bm{p},\bm{\lambda},\bm{a}}\in\Sigma_{\tau}(S)$ and $v\in H_{\bm{p},\bm{\lambda},\bm{a}}$
 $$d^2\ci{J}_{\tau}(\widetilde{\Psi}_{\bm{p},\bm{\lambda},\bm{a}})[v,v]\ge C\N{v}^2.$$
\end{proposizione}

\begin{proof}
 The proof consists of standard computations already present in the literature for the Riemannian case; the only difference is that in this case one needs Lemma 5 from \cite{MU} and the fact that (using the notation from there) the operator $-\Delta_{\H^n}-(Q^*-1)\omega^{Q^*-2}$ has only one negative eigenspace generated by $\omega$ (which is not stated explicitely there, but follows from the proof).
\end{proof}

\begin{lemma}\label{LyapunovSchmidt}
 Given $\widetilde{\Psi}_{\bm{p},\bm{\lambda},\bm{a}}\in\Sigma_{\tau}(S)$ there exists an unique solution $v$ to
 $$\pi_{H_{\bm{p},\bm{\lambda},\bm{a}}}\nabla\ci{J}_{\tau}(\widetilde{\Psi}_{\bm{p},\bm{\lambda},\bm{a}}+v)=0$$
 on $H_{\bm{p},\bm{\lambda},\bm{a}}\cap\{\N{v}<\e\}$, which we will denote by $v_{\bm{p},\bm{\lambda},\bm{a}}$,
 and it verifies the estimate $\N{v_{\bm{p},\bm{\lambda},\bm{a}}}\lesssim \tau^{1/2}$.
\end{lemma}

\begin{proof}
 Let $f\in H_{\bm{p},\bm{\lambda},\bm{a}}$. By Lemma \ref{StimaDifferenza}
 $$\left|\bra f, \Psi_{\bm{p},\bm{\lambda},\bm{a}}\ket\right|\lesssim\tau^{\frac{1}{2}}\N{f}, \left|\bra f, \frac{1}{\lambda^k}\frac{\de\Psi_{\bm{p},\bm{\lambda},\bm{a}}}{\de p^k}\ket\right|\lesssim\tau^{\frac{1}{2}}\N{f}, \left|\bra f, \lambda^k\frac{\de\Psi_{\bm{p},\bm{\lambda},\bm{a}}}{\de\lambda^k}\ket\right|\lesssim\tau^{\frac{1}{2}}\N{f}.$$
 Using this, Lemma \ref{StimaDifferenzaTotale}, the Sobolev inequality \eqref{Sobolev} and Lemmas \ref{StimaIntegralePsi3MenoTau}, \ref{LemmaB3} and \ref{LemmaDifferenzaSublaplaciano},
 $$d\ci{J}_{\tau}(\widetilde{\Psi}_{\bm{p},\bm{\lambda},\bm{a}})[f] = \int_ML_{\theta}\widetilde{\Psi}_{\bm{p},\bm{\lambda},\bm{a}}f - \int_MK\widetilde{\Psi}_{\bm{p},\bm{\lambda},\bm{a}}^{3-\tau}f =$$
 $$=\bra\widetilde{\Psi}_{\bm{p},\bm{\lambda},\bm{a}},f\ket -\int_MK\Psi_{\bm{p},\bm{\lambda},\bm{a}}^{3-\tau}f + \int_MK\left(\Psi_{\bm{p},\bm{\lambda},\bm{a}}^{3-\tau}-\widetilde{\Psi}_{\bm{p},\bm{\lambda},\bm{a}}^{3-\tau}\right)f=$$
 $$= -\int_MK\Psi_{\bm{p},\bm{\lambda},\bm{a}}^{3-\tau}f + O\left(\left(\int_M\left(\Psi_{\bm{p},\bm{\lambda},\bm{a}}^{3-\tau}-\widetilde{\Psi}_{\bm{p},\bm{\lambda},\bm{a}}^{3-\tau}\right)^{\frac{4}{3}}\right)^{\frac{3}{4}}\N{f}_{L^4}\right)=$$
 $$= -\sum_{k=1}^N\int_{B_{\rho}(\overline{x}^k)}Ka_k^{3-\tau}\phi_{p_k,\lambda_k}^{3-\tau}f + O\left(\tau^{\frac{1}{2}}\left(\int_M\left(\Psi_{\bm{p},\bm{\lambda},\bm{a}}^{3-\tau}\right)^{\frac{4}{3}}\right)^{\frac{3}{4}}\N{f}_{H^1}\right)=$$
 $$= -\sum_{k=1}^N\int_{B_{\rho}(\overline{x}^k)}K(p_k)a_k^{3-\tau}\phi_{p_k,\lambda_k}^{3-\tau}f +O\left(\sum_{k=1}^N\int_{B_{\rho}(\overline{x}^k)}|\nabla_b K(p_k)d(x,p_k)a_k^{3-\tau}\phi_{p_k,\lambda_k}^{3-\tau}f\right) + O\left(\tau^{\frac{1}{2}}\right)\N{f}_{H^1}=$$
 $$= \sum_{k=1}^N\int_{B_{\rho}(\overline{x}^k)}K(p_k)a_k^{3-\tau}\left(\phi_{p_k,\lambda_k}^3-\phi_{p_k,\lambda_k}^{3-\tau}\right)f +\sum_{k=1}^N\int_{B_{\rho/2}(\overline{x}^k)}K(p_k)a_k^{3-\tau}\cerchio{\Delta}_b\phi_{p_k,\lambda_k}f+$$
 $$+\sum_{k=1}^N\int_{B_{\rho}(\overline{x}^k)\setminus B_{\rho/2}(\overline{x}^k)}K(p_k)a_k^{3-\tau}\phi_{p_k,\lambda_k}^3f+$$
 $$+O\left(\sum_{k=1}^N\int_{B_{\rho}(\overline{x}^k)}|\nabla_b K(p_k)|d(x,p_k)a_k^{3-\tau}\phi_{p_k,\lambda_k}^{3-\tau}f\right) + O\left(\tau^{\frac{1}{2}}\right)\N{f}_{H^1}=$$
 $$=O\left(\sum_{k=1}^N\left(\int_{B_{\rho}(\overline{x}^k)}\left(\phi_{p_k,\lambda_k}^3-\phi_{p_k,\lambda_k}^{3-\tau}\right)^{\frac{4}{3}}\right)^{\frac{3}{4}}\N{f}_{L^4}\right)+$$
 $$+ \sum_{k=1}^N\int_{B_{\rho}(\overline{x}^k)}K(p_k)a_k^{3-\tau}(\cerchio{\Delta}_b-\Delta_b)\phi_{p_k,\lambda_k}f+ \sum_{k=1}^N\int_{B_{\rho}(\overline{x}^k)}K(p_k)a_k^{3-\tau}\Delta_b\phi_{p_k,\lambda_k}f+$$
 $$-\sum_{k=1}^N\int_{B_{\rho}(\overline{x}^k)\setminus B_{\rho/2}(\overline{x}^k)}K(p_k)a_k^{3-\tau}\cerchio{\Delta}_b\phi_{p_k,\lambda_k}f + O\left(\sum_{k=1}^N\left(\int_{B_{\rho}(\overline{x}^k)\setminus B_{\rho/2}(\overline{x}^k)}\phi_{p_k,\lambda_k}^4\right)^{\frac{3}{4}}\N{f}_{L^4}\right) +$$
 $$+O\left(\sum_{k=1}^N|\nabla_b K(p_k)|\left(\int_{B_{\rho}(\overline{x}^k)}d(x,p_k)^{\frac{4}{3}}\phi_{p_k,\lambda_k}^{4-\frac{4}{3}\tau}\right)^{\frac{3}{4}}\N{f}_{L^4}\right) + O\left(\tau^{\frac{1}{2}}\right)\N{f}_{H^1}=$$
 $$= O\left(\tau|\log\tau|\right)\N{f}_{H^1} +$$
 $$+ O\left(\sum_{k=1}^N\left[\int_{B_{\rho/2}(\overline{x}^k)}\left((\cerchio{\Delta}_b-\Delta_b)\phi_{p_k,\lambda_k}\right)^{\frac{4}{3}}\right]^{\frac{3}{4}}\right)\N{f}_{L^4}+$$
 $$+\sum_{k=1}^N\frac{1}{4}\int_{B_{\rho}(\overline{x}^k)}K(p_k)a_k^{3-\tau}R\phi_{p_k,\lambda_k}f+ O\left(\sum_{k=1}^N\left(\int_{B_{\rho}(\overline{x}^k)\setminus B_{\rho/2}(\overline{x}^k)}(\cerchio{\Delta}_b\phi_{p_k,\lambda_k})^{4/3}\right)^{3/4}\N{f}_{L^4}\right)+$$
 $$+ O\left(\sum_{k=1}^N\left(\lambda_k^{-4}\int_{B_{\lambda_k\rho}(\overline{x}^k)\setminus B_{\lambda_k\rho/2}(\overline{x}^k)}(1+|x\cdot\delta_{\lambda_k}p_k^{-1}|)^{-8}\right)^{\frac{3}{4}}\N{f}_{L^4}\right) +$$
 $$+ O\left(\sum_{k=1}^N|\nabla_b K(p_k)|\left(\lambda_k^{-4}\int_{B_{\lambda_k\rho}(\overline{x}^k)}\lambda_k^{-\frac{4}{3}}|x\cdot\delta_{\lambda_k}p_k^{-1}|^{\frac{4}{3}}\lambda_k^{4-4\tau/3}(1+|x\cdot\delta_{\lambda_k}p_k^{-1}|)^{-8-8\tau/3}\right)^{\frac{3}{4}}\right)\N{f}_{H^1}+$$
 $$+O\left(\tau^{\frac{1}{2}}\right)\N{f}_{H^1}=$$
 $$= O\left(\tau|\log\tau|\right)\N{f}_{H^1} + O\left(\tau^{1/2}\right)\N{f}_{H^1} +$$
 $$+O\left(\sum_{k=1}^N\left[\lambda_k^{-4}\int_{B_{\lambda_k\rho}(\overline{x}^k)}\left(\lambda_k^{-2}|x|^2\lambda_k(1+|x\cdot\delta_{\lambda_k}p_k^{-1}|)^{-2}\right)^{\frac{4}{3}}\right]^{\frac{3}{4}}\N{f}_{L^4}\right) +$$
 $$+O\left(\sum_{k=1}^N\left[\lambda_k^{-4}\int_{B_{\lambda_k\rho}(\overline{x}^k)\setminus B_{\lambda_k\rho/2}(\overline{x}^k)}\left(\lambda_k^2\lambda_k(1+|x\cdot\delta_{\lambda_k}p_k^{-1}|)^{-4}\right)^{\frac{4}{3}}\right]^{\frac{3}{4}}\right)\N{f}_{H^1} +$$
 $$+ O\left(\sum_{k=1}^N\left(\lambda_k^{-4}\lambda_k^{-8}\Vol(B_{\lambda_k\rho}(\overline{x}^k)\setminus B_{\lambda_k\rho/2}(\overline{x}^k))\right)^{\frac{3}{4}}\right)\N{f}_{H^1} +$$
 $$+ O\left(\sum_{k=1}^N|\nabla_b K(p_k)|\left(\lambda_k^{-4}\lambda_k^{-\frac{4}{3}}\lambda_k^{4-4\tau/3}\right)^{\frac{3}{4}}\right)\N{f}_{H^1}=$$
 $$= O(\tau^{1/2})\N{f}_{H^1}+ $$
 $$+ O\left(\sum_{k=1}^N\left[\lambda_k^{-4}\int_{B_{\rho\lambda_i/2}(\overline{x}^k)}\left(\lambda_k^{-1} + \lambda_k^{-5}(1+|x|)^2\right)^{\frac{4}{3}}\right]^{\frac{3}{4}}\right)\N{f}_{H^1} +$$
 $$+O\left(\sum_{k=1}^N\lambda_k^{-4}\Vol(B_{\lambda_k\rho}(\overline{x}^k))^{\frac{3}{4}}\right)\N{f}_{H^1} +$$
 $$+O\left(\sum_{k=1}^N\left[\lambda_k^{-\frac{16}{3}}\Vol(B_{\lambda_k\rho}(\overline{x}^k)\setminus B_{\lambda_k\rho/2}(\overline{x}^k))\right]^{\frac{3}{4}}\right)\N{f}_{H^1} +$$
 $$+O\left(\sum_{k=1}^N\lambda_k^{-6}\right)\N{f}_{H^1} + O\left(\sum_{k=1}^N\lambda_k^{-1}\right)\N{f}_{H^1}=$$
 $$= O(\tau^{1/2})\N{f}_{H^1}+ O\left(\sum_{k=1}^N\lambda_k^{-3}\lambda_k^{-1}\left[\Vol(B_{\rho\lambda_i/2}(\overline{x}^k))\right]^{\frac{3}{4}}\right)\N{f}_{H^1} +$$
 $$+O\left(\tau^{\frac{1}{2}}\right)\N{f}_{H^1} +O\left(\tau^{\frac{1}{2}}\right)\N{f}_{H^1}+O\left(\tau^3\right)\N{f}_{H^1}+O\left(\tau^{\frac{1}{2}}\right)\N{f}_{H^1}+O\left(\tau^{\frac{1}{2}}\right)\N{f}_{H^1}=$$
 $$=O\left(\tau^{\frac{1}{2}}\right)\N{f}_{H^1}.$$
 Therefore $\N{\pi_{H_{\bm{p},\bm{\lambda},\bm{a}}}\nabla\ci{J}_{\tau}(\widetilde{\Psi}_{\bm{p},\bm{\lambda},\bm{a}})}\lesssim\tau^{1/2}$.
 The equation we want to solve,
 $$\pi_{H_{\bm{p},\bm{\lambda},\bm{a}}}\nabla\ci{J}_{\tau}(\widetilde{\Psi}_{\bm{p},\bm{\lambda},\bm{a}}+v)=0,$$
 can be written as
 $$\pi_{H_{\bm{p},\bm{\lambda},\bm{a}}}\left(\nabla\ci{J}_{\tau}(\widetilde{\Psi}_{\bm{p},\bm{\lambda},\bm{a}})\right) + \left(\pi_{H_{\bm{p},\bm{\lambda},\bm{a}}}\circ\ci{J}''_{\tau}(\widetilde{\Psi}_{\bm{p},\bm{\lambda},\bm{a}})\right)v + R(\widetilde{\Psi}_{\bm{p},\bm{\lambda},\bm{a}},v)=0.$$
 where $R$ is a remainder term.
 Since we are looking for a solution $v\in H_{\bm{p},\bm{\lambda},\bm{a}}$ and by Proposition \ref{DifferenzialeSecondoPositivo} $\pi_{H_{\bm{p},\bm{\lambda},\bm{a}}}\circ\ci{J}''_{\tau}(\widetilde{\Psi}_{\bm{p},\bm{\lambda},\bm{a}})$ is invertible on this subspace, we can write the equation as
 $$v= -\left(\pi_{H_{\bm{p},\bm{\lambda},\bm{a}}}\circ\ci{J}''_{\tau}(\widetilde{\Psi}_{\bm{p},\bm{\lambda},\bm{a}})\right)^{-1}\left(\pi_{H_{\bm{p},\bm{\lambda},\bm{a}}}\left(\nabla\ci{J}_{\tau}(\widetilde{\Psi}_{\bm{p},\bm{\lambda},\bm{a}})\right) + R(\widetilde{\Psi}_{\bm{p},\bm{\lambda},\bm{a}},v)\right):= N(v).$$
 Since, by Proposition \ref{DifferenzialeSecondoPositivo}, $\left(\pi_{H_{\bm{p},\bm{\lambda},\bm{a}}}\circ\ci{J}''_{\tau}(\widetilde{\Psi}_{\bm{p},\bm{\lambda},\bm{a}})\right)^{-1}$ is uniformly bounded, the existence of a solution follows from a contraction argument.
 By the above estimate on $\N{\pi_{H_{\bm{p},\bm{\lambda},\bm{a}}}\nabla\ci{J}_{\tau}(\Psi)}$ there exists $C$ such that the ball $B_{C\tau^{\frac{1}{2}}}\cap H_{\bm{p},\bm{\lambda},\bm{a}}$ is invariant, and using the contraction theorem thereon allows to prove the estimate stated in the lemma.
\end{proof}

\begin{lemma}\label{Lemmaa}
 If $\widetilde{\Psi}_{\bm{p},\bm{\lambda},\bm{a}}+v\in V_{\tau}(S)$ is a critical point of $d\ci{J}_{\tau}$, then $a_k-\frac{2}{K(x_k)^{1/2}}=O(\tau^{1/2})$.
\end{lemma}

\begin{proof}
 By Lemma \ref{LyapunovSchmidt} $\N{v}\lesssim\tau^{1/2}$, and by Lemma \ref{StimaDifferenza} $\N{\Psi_{\bm{p},\bm{\lambda},\bm{a}}-\widetilde{\Psi}_{\bm{p},\bm{\lambda},\bm{a}}}\lesssim\tau^{1/2}$, so we can write
 $$\widetilde{\Psi}_{\bm{p},\bm{\lambda},\bm{a}}+v = \Psi_{\bm{p},\bm{\lambda},\bm{a}} + \widetilde{v}$$
 with $\N{\widetilde{v}}\lesssim\tau^{1/2}$.
 
 Therefore
 $$0=d\ci{J}_{\tau}\left(\Psi_{\bm{p},\bm{\lambda},\bm{a}} + \widetilde{v}\right)\left[\phi_{p_k,\lambda_k}\right]=$$
 $$=\int_M\left(\Psi_{\bm{p},\bm{\lambda},\bm{a}} + \widetilde{v}\right)L_{\theta}\phi_{p_k,\lambda_k} -\int_MK\left|\Psi_{\bm{p},\bm{\lambda},\bm{a}} + \widetilde{v}\right|^{2-\tau}\left(\Psi_{\bm{p},\bm{\lambda},\bm{a}} + \widetilde{v}\right)\phi_{p_k,\lambda_k}.$$
 Since, if $x>0$, $||x+y|^{2-\tau}(x+y) - x^{3-\tau}| \lesssim x^{2-\tau}|y| + |y|^{3-\tau}$, and using Lemma \ref{LemmaDifferenzaSublaplaciano},
 $$0=\int_{B_{\rho}(\overline{x}^k)}a_k\phi_{p_k,\lambda_k}L_{\theta}\phi_{p_k,\lambda_k} +\bra\widetilde{v},\phi_{p_k,\lambda_k}\ket -\int_{B_{\rho}(\overline{x}^k)}K\left|\Psi_{\bm{p},\bm{\lambda},\bm{a}}\right|^{3-\tau}\phi_{p_k,\lambda_k}+$$
 $$ +O\left(\int_{B_{\rho}(\overline{x}^k)}(\left|\Psi_{\bm{p},\bm{\lambda},\bm{a}}\right|^{2-\tau}|v| + |v|^{3-\tau})\phi_{p_k,\lambda_k}\right) = $$
 $$= \int_{B_{\rho/2}(\overline{x}^k)}4a_k\phi_{p_k,\lambda_k}^4 +\int_{B_{\rho/2}(\overline{x}^k)}a_kR\phi_{p_k,\lambda_k}^2 - \int_{B_{\rho}(\overline{x}^k)\setminus B_{\rho/2}(\overline{x}^k)}\phi_{p_k,\lambda_k}4\cerchio{\Delta_b}\phi_{p_k,\lambda_k}+ $$
 $$+ \int_{B_{\rho}(\overline{x}^k)}4a_k\phi_{p_k,\lambda_k}(\Delta_b-\cerchio{\Delta_b})\phi_{p_k,\lambda_k} +O(\tau^{1/2}) +$$
 $$-\int_{B_{\rho}(\overline{x}^k)}Ka_k^{3-\tau}\phi_{p_k,\lambda_k}^{4-\tau} +O\left(\int_{B_{\rho}(\overline{x}^k)}(\phi_{p_k,\lambda_k}^{3-\tau}|v| + \phi_{p_k,\lambda_k}|v|^{3-\tau}) \right) = $$
 $$= \int_{B_{\rho/2}(\overline{x}^k)}4a_k\phi_{p_k,\lambda_k}^4+ O\left(\lambda_k^{-4}\int_{B_{\rho\lambda_k/2}(\overline{x}^k)}\lambda_k^2(1+|x\cdot\delta_{\lambda_i^{-1}}p_i^{-1}|)^{-4}\right)+$$
 $$+O\left(\lambda_k^{-4}\int_{B_{\lambda_k\rho}(\overline{x}^k)\setminus B_{\lambda_k\rho/2}(\overline{x}^k)}\lambda_k(1+|x\cdot\delta_{\lambda_k^{-1}}p_k^{-1}|)^{-2}\lambda_k\lambda_k^2(1+|x\cdot\delta_{\lambda_k^{-1}}p_k^{-1}|)^{-4}\right) +$$
 $$+ O\left(\int_{B_{\rho/2}(\overline{x}^k)}\phi_{p_k,\lambda_k}(\Delta_b-\cerchio{\Delta_b})\phi_{p_k,\lambda_k}\right) +$$
 $$+O(\tau^{1/2}) -\int_{B_{\rho}(\overline{x}^k)}K(p_k)a_k^{3-\tau}\phi_{p_k,\lambda_k}^{4-\tau} + O\left(\int_{B_{\rho}(\overline{x}^k)}|\nabla K(p_k)|d(x,p_k)a_k^{3-\tau}\phi_{p_k,\lambda_k}^{4-\tau}\right) +$$
 $$+O\left(\left(\int_{B_{\rho}(\overline{x}^k)}\left(\phi_{p_k,\lambda_k}^{3-\tau}\right)^{\frac{4}{3}}\right)^{\frac{3}{4}}\N{v}_{L^4} + \left(\int_{B_{\rho}(\overline{x}^k)}\left(\phi_{p_k,\lambda_k}\right)^{\frac{4}{1+\tau}}\right)^{\frac{1+\tau}{4}}\N{v}_{L^4}\right) = $$
 $$= \left(4a_k-K(p_k)a_k^3\right)\int_{B_{\rho/2}(\overline{x}^k)}\phi_{p_k,\lambda_k}^4+ \int_{B_{\rho/2}(\overline{x}^k)}\left(K(p_k)a_k^3\phi_{p_k,\lambda_k}^4-K(p_k)a_k^{3-\tau}\phi_{p_k,\lambda_k}^{4-\tau}\right)+$$
 $$-\int_{B_{\rho}(\overline{x}^k)\setminus B_{\rho/2}(\overline{x}^k)}K(p_k)a_k^{3-\tau}\phi_{p_k,\lambda_k}^{4-\tau}+ O(\lambda_k^{-2}\log\lambda_k)+ O\left(\lambda_k^{-6}\Vol(B_{\lambda_k\rho}(\overline{x}^k)\setminus B_{\lambda_k\rho/2}(\overline{x}^k))\right)+$$
 $$+O(\tau^{1/2})+ O(\tau^{1/2})+$$
 $$+O\left(|\nabla K(p_k)|\lambda_k^{-4}\int_{B_{\lambda_k\rho}(\overline{x}^k)}\lambda_k^{-1}|x\cdot\delta_{\lambda_k^{-1}}p_k^{-1}|\lambda_k^{4-\tau}(1+|x\cdot\delta_{\lambda_k^{-1}}p_k^{-1}|)^{-8+2\tau}\right) +$$
 $$+O(\N{v})=$$
 $$= \left(4a_k-K(p_k)a_k^3\right)\lambda_k^{-4}\int_{B_{\lambda_k\rho/2}(0)}\lambda_k^4U^4+ \int_{B_{\rho/2}(\overline{x}^k)}\left(K(p_k)a_k^3\phi_{p_k,\lambda_k}^4-K(p_k)a_k^3\phi_{p_k,\lambda_k}^{4-\tau}\right)+$$
 $$+ O\left(\tau\int_{B_{\rho/2}(\overline{x}^k)}K(p_k)a_k^{3-\tau}\phi_{p_k,\lambda_k}^{4-\tau}\right)+$$
 $$+O\left(\lambda_k^{-4}\int_{B_{\rho\lambda_k}(\overline{x}^k)\setminus B_{\rho\lambda_k/2}(\overline{x}^k)}a_k^{3-\tau}\left(\lambda_k(1+|x\cdot\delta_{\lambda_k^{-1}}p_k^{-1}|)^{-2}\right)^{4-\tau}\right)+ O(\tau\log\tau)+$$
 $$+O(\tau)+$$
 $$+ O(\tau^{1/2}) + O\left(|\nabla K(p_k)|\lambda_k^{-1}\lambda_k^{-\tau}\right)+ O(\tau^{1/2})=$$
 $$= \left(4a_k-K(p_k)a_k^3\right)\left(\int_{\H^1}U^4+o(1)\right)+$$
 $$+O\left(\lambda_i^{-4}\int_{B_{\lambda_i\rho/2}(\overline{x}^k)}(\lambda_i^4(1+|x\cdot\delta_{\lambda_k^{-1}}p_k^{-1}|)^{-8} -\lambda_k^{4-\tau}(1+|x\cdot\delta_{\lambda_k^{-1}}p_k^{-1}|)^{-8+2\tau})\right) +$$
 $$+O\left(\tau\lambda_k^{-4}\int_{B_{\lambda_k\rho/2}(\overline{x}^k)}\lambda_k^{4-\tau}(1+|x\cdot\delta_{\lambda_k^{-1}}p_k^{-1}|)^{-8+2\tau}\right)+$$
 $$+O\left(\lambda_k^{-\tau}\lambda_k^{-8+2\tau}\Vol(B_{\rho\lambda_k}(\overline{x}^k)\setminus B_{\rho\lambda_k/2}(\overline{x}^k))\right)+ O(\tau^{1/2})=$$
 $$= \left(4a_k-K(p_k)a_k^3\right)\left(\int_{\H^1}U^4+o(1)\right)+$$
 $$+O\left(\tau|\log\tau|\int_{B_{\lambda_k\rho/2}(\overline{x}^k)}(1+|x\cdot\delta_{\lambda_k^{-1}}p_k^{-1}|)^{-8}\right)+O(\tau)+ O(\lambda_k^{-4}) + O(\tau^{1/2}) = $$
 $$= \left(4a_k-K(p_k)a_k^3\right)\left(\int_{\H^1}U^4+o(1)\right)+O(\tau^{1/2}).$$
 Since $\left|a_k-\frac{2}{K(x_k)^{1/2}}\right|<\e$ this implies the thesis.
\end{proof}

Thus if we define
$$\Sigma'_{\tau}(S) = \left\{\widetilde{\Psi}_{\bm{p},\bm{\lambda},\bm{a}}\;\middle|\;d(p_k,\overline{x}^k)<A\tau^{\frac{1}{4}},A^{-1}\tau^{-\frac{1}{2}}<\lambda_i<A\tau^{-\frac{1}{2}},\left|a_k-\frac{2}{K(\overline{x}^k)^{\frac{1}{2}}}\right|<A\tau^{\frac{1}{2}}|\log\tau|\right\}$$
and
\begin{equation}\label{DefinizioneVtauprimo}
 V'_{\tau}(S)= \left\{\widetilde{\Psi}_{\bm{p},\bm{\lambda},\bm{a}}+v\;\middle|\; \widetilde{\Psi}_{\bm{p},\bm{\lambda},\bm{a}}\in\Sigma'_{\tau}(S), \N{v}_{H^1}<A\tau^{1/2},v\in H_{\bm{p},\bm{\lambda},\bm{a}}\right\}
\end{equation}
then Proposition \ref{DicotomiaSoluzioni} holds with $V'_{\tau}(S)$ in place of $U_{\tau}(S)$.

\begin{proposizione}\label{GradoPuntiCriticiInfinito}
 $$\deg(\nabla\ci{J}_{\tau},V'_{\tau}(S),0)=(-1)^{\sum_{k=1}^N\ind(K,\overline{x}^k)}$$
\end{proposizione}

\begin{proof}
 We want to use Proposition \ref{GradoProdotto}.
 If we define $\Omega_k= B_{A\tau^{\frac{1}{4}}}(p_k)$,
 $\Omega_{N+k}=(A^{-1}\tau^{-\frac{1}{2}},A\tau^{-\frac{1}{2}})$,
 $\Omega_{2N+k}=\left(\frac{2}{K(\overline{x}^k)^{\frac{1}{2}}}-A\tau^{\frac{1}{2}}\log\tau,\frac{2}{K(\overline{x}^k)^{\frac{1}{2}}}+A\tau^{\frac{1}{2}}\log\tau\right)$,
 and $\Omega_{3N+1}=B_{A\tau^{1/2}}(0)\subset H_{\bm{p},\bm{\lambda},\bm{a}}$, then $V'_{\tau}(S)$ is in natural correspondence with $\bigoplus_{i=1}^{3N+1}\Omega_i$.
 $B_{A\tau^{\frac{1}{4}}}(p_k)$ is not a subset of an Euclidean space, but it is a subset of a small neighborhood of an oriented manifold, so Proposition \ref{GradoProdotto} applies with minor notational changes.
 
 Let us compute
 $$\frac{\de}{\de a_k}\ci{J}_{\tau}(\widetilde{\Psi}_{\bm{p},\bm{\lambda},\bm{a}}+v) =$$
 $$=\int_M\frac{\de}{\de a_k}\widetilde{\Psi}_{\bm{p},\bm{\lambda},\bm{a}}L_{\theta}(\widetilde{\Psi}_{\bm{p},\bm{\lambda},\bm{a}}+v)) - \int_MK(\widetilde{\Psi}_{\bm{p},\bm{\lambda},\bm{a}}+v)^{3-\tau}\frac{\de}{\de a_k}\widetilde{\Psi}_{\bm{p},\bm{\lambda},\bm{a}}.$$
 It can be easily verified that the contribute of the terms involving $v$ is $O(\tau^{1/2})$, and that when replacing $\frac{\de}{\de a_k}\widetilde{\Psi}_{\bm{p},\bm{\lambda},\bm{a}}$ with $\frac{\de}{\de a_k}\Psi_{\bm{p},\bm{\lambda},\bm{a}}$ and $\widetilde{\Psi}_{\bm{p},\bm{\lambda},\bm{a}}^{3-\tau}$ with $\Psi_{\bm{p},\bm{\lambda},\bm{a}}^{3-\tau}$ the above quantity is modified by $O(\tau^{1/2})$, so
 $$\frac{\de}{\de a_k}\ci{J}_{\tau}(\widetilde{\Psi}_{\bm{p},\bm{\lambda},\bm{a}}+v) =$$
 $$=\int_M\frac{\de}{\de a_k}\Psi_{\bm{p},\bm{\lambda},\bm{a}}L_{\theta}\widetilde{\Psi}_{\bm{p},\bm{\lambda},\bm{a}} - \int_MK\Psi_{\bm{p},\bm{\lambda},\bm{a}}^{3-\tau}\frac{\de}{\de a_i}\Psi_{\bm{p},\bm{\lambda},\bm{a}} + O(\tau^{1/2}).$$
 By definition $\frac{\de}{\de a_k}\Psi_{\bm{p},\bm{\lambda},\bm{a}}=\phi_{p_k,\lambda_k}$, so
 $$\frac{\de}{\de a_k}\ci{J}_{\tau}(\widetilde{\Psi}_{\bm{p},\bm{\lambda},\bm{a}}+v) =\int_M\phi_{p_k,\lambda_k}L_{\theta}\widetilde{\Psi}_{\bm{p},\bm{\lambda},\bm{a}} - \int_MK\Psi_{\bm{p},\bm{\lambda},\bm{a}}^{3-\tau}\phi_{p_k,\lambda_k} + O(\tau^{1/2})=$$
 $$= \int_M\phi_{p_k,\lambda_k}4a_k\phi_{p_k,\lambda_k}^3 - \int_MKa_k^{3-\tau}\phi_{p_k,\lambda_k}^{3-\tau}\phi_{p_k,\lambda_k} + O(\tau^{1/2})=$$
 $$= 4a_k\int_{B_{\rho}(x_k)}\phi_{p_k,\lambda_k}^4 - a_k^{3-\tau}K(x_k)\int_{B_{\rho}(x_k)}\phi_{p_k,\lambda_k}^{4-\tau} - a_k^{3-\tau}\int_{B_{\rho}(x_k)}(K-K(x_k))\phi_{p_k,\lambda_k}^{4-\tau} + O(\tau^{1/2})=$$
 $$= 4a_k\int_{B_{\rho}(x_k)}\phi_{p_k,\lambda_k}^4 - a_k^{3-\tau}K(x_k)\int_{B_{\rho}(x_k)}\phi_{p_k,\lambda_k}^4 - a_k^{3-\tau}K(x_k)\int_{B_{\rho}(x_k)}(\phi_{p_k,\lambda_k}^{4-\tau}-\phi_{p_k,\lambda_k}^4)+$$
 $$+O\left(\int_{B_{\rho}(x_k)}d(x,x_k)\phi_{p_k,\lambda_k}^{4-\tau}\right) + O(\tau^{1/2})=$$
 $$= \left(4a_k- a_k^{3-\tau}K(x_k)\right)\int_{B_{\rho}(x_k)}\phi_{p_k,\lambda_k}^4 +O\left(\tau\log\tau\right)+O\left(\tau^{1/2}\right) + O(\tau^{1/2})=$$
 $$=(C+o(1))\left(4a_k- a_k^{3-\tau}K(x_k)\right)+ O(\tau^{1/2}).$$
 Therefore if $a_k-\frac{2}{K(x_k)^{1/2}}=A\tau^{1/2}\log\tau$ then $\frac{\de}{\de a_k}\ci{J}_{\tau}(\widetilde{\Psi}_{\bm{p},\bm{\lambda},\bm{a}}+v) <0$, while if $a_k-\frac{2}{K(x_k)^{1/2}}=-A\tau^{1/2}\log\tau$ then $\frac{\de}{\de a_k}\ci{J}_{\tau}(\widetilde{\Psi}_{\bm{p},\bm{\lambda},\bm{a}}+v) >0$.
 Therefore, viewing $\frac{\de}{\de a_k}\ci{J}_{\tau}(\widetilde{\Psi}_{\bm{p},\bm{\lambda},\bm{a}}+v)$ as a function of $a_k$,
 \begin{equation}\label{GradoDimostrazione1}
  \deg\left(\frac{\de}{\de a_k}\ci{J}_{\tau}(\widetilde{\Psi}_{\bm{p},\bm{\lambda},\bm{a}}+v),\Omega_{2N+k},0\right)=-1
 \end{equation}
 Now, viewing $\ci{J}_{\tau}(\widetilde{\Psi}_{\bm{p},\bm{\lambda},\bm{a}}+v)$ as a function of $p_k$, let us derive it with respect to the vector field $X$:
 $$X_{p_k}\ci{J}_{\tau}(\widetilde{\Psi}_{\bm{p},\bm{\lambda},\bm{a}}+v) =$$
 $$=\int_MX_{p_k}\widetilde{\Psi}_{\bm{p},\bm{\lambda},\bm{a}}L_{\theta}(\widetilde{\Psi}_{\bm{p},\bm{\lambda},\bm{a}}+v)) - \int_MK(\widetilde{\Psi}_{\bm{p},\bm{\lambda},\bm{a}}+v)^{3-\tau}X_{p_k}\widetilde{\Psi}_{\bm{p},\bm{\lambda},\bm{a}}=$$
 $$=\int_MX_{p_k}\widetilde{\Psi}_{\bm{p},\bm{\lambda},\bm{a}}L_{\theta}\widetilde{\Psi}_{\bm{p},\bm{\lambda},\bm{a}} - \int_MK(\widetilde{\Psi}_{\bm{p},\bm{\lambda},\bm{a}}+v)^{3-\tau}X_{p_k}\widetilde{\Psi}_{\bm{p},\bm{\lambda},\bm{a}}$$
 where we used the fact that $v\in H_{\bm{\bm{p},\bm{\lambda},\bm{a}}}$.
 It can be easily verified that the contribute of the term involving $v$ is $O(\tau^{1/2})$, and that when replacing $X_{p_k}\widetilde{\Psi}_{\bm{p},\bm{\lambda},\bm{a}}$ with $X_{p_k}\Psi_{\bm{p},\bm{\lambda},\bm{a}}$ and $\widetilde{\Psi}_{\bm{p},\bm{\lambda},\bm{a}}^{3-\tau}$ with $\Psi_{\bm{p},\bm{\lambda},\bm{a}}^{3-\tau}$ the above quantity is modified by $O(\tau^{1/2})$, so
 $$X_{p_k}\ci{J}_{\tau}(\widetilde{\Psi}_{\bm{p},\bm{\lambda},\bm{a}}+v) =$$
 $$=\int_MX_{p_k}\Psi_{\bm{p},\bm{\lambda},\bm{a}}L_{\theta}\widetilde{\Psi}_{\bm{p},\bm{\lambda},\bm{a}} - \int_MK\Psi_{\bm{p},\bm{\lambda},\bm{a}}^{3-\tau}X_{p_k}\Psi_{\bm{p},\bm{\lambda},\bm{a}}+ O(\tau^{1/2})=$$
 $$=\int_Ma_kX_{p_k}\phi_{p_k,\lambda_k}a_k\phi_{p_k,\lambda_k}^3 - \int_MKa_k^{3-\tau}\phi_{p_k,\lambda_k}^{3-\tau}a_kX_{p_k}\phi_{p_k,\lambda_k}+ O(\tau^{1/2})=$$
 $$=\frac{1}{4}a_k^2\int_MX_{p_k}(\phi_{p_k,\lambda_k}^4) - \frac{1}{4-\tau}a_k^{4-\tau}\int_MKX_{p_k}(\phi_{p_k,\lambda_k}^{4-\tau})+ O(\tau^{1/2}).$$
 In $B_{\rho/2}(\overline{x}^k)$ by definition $X_{p_k}(\phi_{p_k,\lambda_k}^s(x))=-X_x(\phi_{p_k,\lambda_k}^s(x))$ for any $s\in\R$, so
 $$X_{p_k}\ci{J}_{\tau}(\widetilde{\Psi}_{\bm{p},\bm{\lambda},\bm{a}}+v) =$$
 $$=-\frac{1}{4}a_k^2\int_{B_{\rho/2}(\overline{x}^k)}X_x(\phi_{p_k,\lambda_k}^4) +\frac{1}{4}a_k^2\int_{M\setminus B_{\rho/2}(\overline{x}^k)}X_{p_k}(\phi_{p_k,\lambda_k}^4) + \frac{1}{4-\tau}a_k^{4-\tau}\int_{B_{\rho/2}(\overline{x}^k)}KX_x(\phi_{p_k,\lambda_k}^{4-\tau})+$$
 $$-\frac{1}{4-\tau}a_k^{4-\tau}\int_{M\setminus B_{\rho/2}(\overline{x}^k)}KX_{p_k}(\phi_{p_k,\lambda_k}^{4-\tau}) + O(\tau^{1/2})=$$
 $$= \frac{1}{4}a_k^2\int_{B_{\rho/2}(\overline{x}^k)}\phi_{p_k,\lambda_k}^4\operatorname{div}X_x -\frac{1}{4}a_k^2\int_{\de B_{\rho/2}(\overline{x}^k)}\phi_{p_k,\lambda_k}^4X_x\cdot\ni + $$
 $$+O\left(\int_{M\setminus B_{\rho/2}(\overline{x}^k)}\tau^{1/2}\lambda_k^{-4}d(x,p_k)^{-9}\right) -\frac{1}{4-\tau}a_k^{4-\tau}\int_{B_{\rho/2}(\overline{x}^k)}\phi_{p_k,\lambda_k}^{4-\tau}X_xK + $$
 $$+ \frac{1}{4-\tau}a_k^{4-\tau}\int_{B_{\rho/2}(\overline{x}^k)}K\phi_{p_k,\lambda_k}^{4-\tau}\operatorname{div}X_x - \frac{1}{4-\tau}a_k^{4-\tau}\int_{\de B_{\rho/2}(\overline{x}^k)}K\phi_{p_k,\lambda_k}^{4-\tau}X_x\cdot\ni+$$
 $$+O\left(\int_{M\setminus B_{\rho/2}(\overline{x}^k)}\tau^{1/2}\lambda_k^{-4}d(x,p_k)^{-9+2\tau}\right) + O(\tau^{1/2})=$$
 $$= O\left(\lambda_k^{-4}\int_{B_{\lambda_k\rho/2}(\overline{x}^k)}\lambda_k^4(1+|x\cdot\delta_{\lambda_k}p_k|)^{-8}\lambda_k^{-3}d(x,\overline{x}^k)^3\right) +$$
 $$+O\left(\int_{\de B_{\rho/2}(\overline{x}^k)}\lambda_k^{-4}d(x,p_k)^{-8}\right) + O(\tau^{5/2})+$$
 $$-\frac{1}{4-\tau}a_k^{4-\tau}\lambda_k^{-4}\int_{B_{\lambda_k\rho/2}(\overline{x}^k)}\lambda_k^{4-\tau}(1+|x\cdot\delta_{\lambda_k}p_k|)^{-8+2\tau}(X_xK(p_k)+ O(\lambda_k^{-1}|x\cdot\delta_{\lambda_k}p_k|))+$$
 $$+O\left(\lambda_k^{-4}\int_{B_{\lambda_k\rho/2}(\overline{x}^k)}\lambda_k^4(1+|x\cdot\delta_{\lambda_k}p_k|)^{-8+2\tau}\lambda_k^{-3}d(x,\overline{x}^k)^3\right) +$$
 $$+O\left(\int_{\de B_{\rho/2}(\overline{x}^k)}\lambda_k^{-4+\tau}d(x,p_k)^{-8+2\tau}\right)+O(\tau^{5/2})=$$
 $$= O(\tau^{3/2}) + O(\tau^2) +O(\tau^{5/2})+$$
 $$-(1+o(1))\frac{1}{4}a_k^4X_xK(p_k)\left(\int_{B_{\lambda_k\rho/2}(\overline{x}^k)}(1+|x\cdot\delta_{\lambda_k}p_k|)^{-8+2\tau}\right)
 =$$
 $$= -C(1+o(1))X_xK(p_k) + o(1).$$

 Performing analogous computations for $YK$ and $TK$, we derive that
 \begin{equation}\label{GradoDimostrazione2}
  \deg((X,Y,T)\ci{J}_{\tau},\Omega_k,0) = \deg(-\nabla_{g_{\theta}}K,\Omega_k,0)=(-1)^{3-\operatorname{ind}(K,\overline{x}_k)}
 \end{equation}
 where the second equality is classical in degree theory.
 
 Now compute
 $$\frac{\de}{\de\lambda_k}\ci{J}_{\tau}(\widetilde{\Psi}_{\bm{p},\bm{\lambda},\bm{a}}+v) =$$
 $$=\int_M(\widetilde{\Psi}_{\bm{p},\bm{\lambda},\bm{a}}+v)\frac{\de}{\de\lambda_k}L_{\theta}\widetilde{\Psi}_{\bm{p},\bm{\lambda},\bm{a}} - \int_MK(\widetilde{\Psi}_{\bm{p},\bm{\lambda},\bm{a}}+v)^{3-\tau}\frac{\de}{\de\lambda_k}\widetilde{\Psi}_{\bm{p},\bm{\lambda},\bm{a}}.$$
 It is easy to verify that the terms involving $v$ give a contribute of $O(\tau^2)$, so
 $$\frac{\de}{\de\lambda_k}\ci{J}_{\tau}(\widetilde{\Psi}_{\bm{p},\bm{\lambda},\bm{a}}+v) =$$
 $$=\int_M\widetilde{\Psi}_{\bm{p},\bm{\lambda},\bm{a}}\frac{\de}{\de\lambda_k}\Psi_{\bm{p},\bm{\lambda},\bm{a}}^3 - \int_MK\widetilde{\Psi}_{\bm{p},\bm{\lambda},\bm{a}}^{3-\tau}\frac{\de}{\de\lambda_k}\widetilde{\Psi}_{\bm{p},\bm{\lambda},\bm{a}} + O(\tau^2)=$$
 $$=\int_{B_{\rho}(\overline{x}_k)}\widetilde{\Psi}_{\bm{p},\bm{\lambda},\bm{a}}a_k\frac{\de}{\de\lambda_k}\phi_{p_k,\lambda_k}^3 - a_k\int_{B_{\rho}(\overline{x}_k)}K\widetilde{\Psi}_{\bm{p},\bm{\lambda},\bm{a}}^{3-\tau}\frac{\de}{\de\lambda_k}\phi_{p_k,\lambda_k} + O(\tau^2)=$$
 $$=\int_{B_{\rho}(\overline{x}_k)}a_k\phi_{p_k,\lambda_k}a_k\frac{\de}{\de\lambda_k}\phi_{p_k,\lambda_k}^3 + a_k\int_{B_{\rho}(\overline{x}_k)}(\widetilde{\Psi}_{\bm{p},\bm{\lambda},\bm{a}}-\Psi_{\bm{p},\bm{\lambda},\bm{a}})\frac{\de}{\de\lambda_k}\phi_{p_k,\lambda_k}^3 +$$
 $$- a_k\int_{B_{\rho}(\overline{x}_k)}K(a_k\phi_{p_k,\lambda_k})^{3-\tau}\frac{\de}{\de\lambda_k}\phi_{p_k,\lambda_k} - a_k\int_{B_{\rho}(\overline{x}_k)}K(\widetilde{\Psi}_{\bm{p},\bm{\lambda},\bm{a}}-\Psi_{\bm{p},\bm{\lambda},\bm{a}})^{3-\tau}\frac{\de}{\de\lambda_k}\phi_{p_k,\lambda_k} +$$
 $$+O\left(\int_{B_{\rho}(\overline{x}_k)}(a_k\phi_{p_k,\lambda_k})^{2-\tau}(\widetilde{\Psi}_{\bm{p},\bm{\lambda},\bm{a}}-\Psi_{\bm{p},\bm{\lambda},\bm{a}})+ (a_k\phi_{p_k,\lambda_k})(\widetilde{\Psi}_{\bm{p},\bm{\lambda},\bm{a}}-\Psi_{\bm{p},\bm{\lambda},\bm{a}})^{2-\tau}\right) +O(\tau^2)=$$
 $$= \frac{1}{4}a_k^2\frac{\de}{\de\lambda_k}\int_{B_{\rho}(\overline{x}_k)}\phi_{p_k,\lambda_k}^4+$$
 $$+ a_k^3\int_{B_{\rho}(\overline{x}_k)}(\widetilde{\Psi}_{\bm{p},\bm{\lambda},\bm{a}}-\Psi_{\bm{p},\bm{\lambda},\bm{a}})3\lambda_k^2U_{p_k}^3\circ\delta_{\lambda_k}\chi_{\overline{x}_k} +$$
 $$+ a_k^3\int_{B_{\rho}(\overline{x}_k)}(\widetilde{\Psi}_{\bm{p},\bm{\lambda},\bm{a}}-\Psi_{\bm{p},\bm{\lambda},\bm{a}})\lambda_k^3\frac{\de}{\de\lambda_k}(U_{p_k}^3\circ\delta_{\lambda_k})\chi_{\overline{x}_k} +$$
 $$- a_k^{4-\tau}\frac{1}{4-\tau}\int_{B_{\rho}(\overline{x}_k)}K\frac{\de}{\de\lambda_k}(\phi_{p_k,\lambda_k}^{4-\tau}) +$$
 $$- a_k\int_{B_{\rho}(\overline{x}_k)}K(\widetilde{\Psi}_{\bm{p},\bm{\lambda},\bm{a}}-\Psi_{\bm{p},\bm{\lambda},\bm{a}})^{3-\tau}U_{p_k}\circ\delta_{\lambda_k}\chi_{\overline{x}_k}+$$
 $$- a_k\int_{B_{\rho}(\overline{x}_k)}K(\widetilde{\Psi}_{\bm{p},\bm{\lambda},\bm{a}}-\Psi_{\bm{p},\bm{\lambda},\bm{a}})^{3-\tau}\lambda_k\frac{\de}{\de\lambda_k}(U_{p_k}\circ\delta_{\lambda_k}\chi_{\overline{x}_k}) =$$
 $$=O(\tau^2) + 3a_k^3 \lambda_k^{-4}\int_{B_{\lambda_k\rho}(\overline{x}_k)}\left(a_k\frac{1}{\lambda_k}A_{p_k}+\sum_{j\ne k}a_j\frac{1}{\lambda_j}G_{\overline{x}_j}(p_k) + O(\tau^{1/2}|x|)\right)\circ\delta_{\lambda_k^{-1}}\cdot\lambda_k^2U_{p_k}^3\chi_{\overline{x}_k} +$$
 $$+ a_k^3\int_{B_{\rho}(\overline{x}_k)}(\widetilde{\Psi}_{\bm{p},\bm{\lambda},\bm{a}}-\Psi_{\bm{p},\bm{\lambda},\bm{a}})\lambda_k^{2}\Xi(U_{p_k}^3)\circ\delta_{\lambda_k}\chi_{\overline{x}_k} +$$
 $$- a_k^{4-\tau}\int_{B_{\rho}(\overline{x}_k)}K\lambda_k^{3-\tau}U^{4-\tau}_{p_k}\circ\delta_{\lambda_k}\chi_{\overline{x}_k} +$$
 $$- a_k^{4-\tau}\frac{1}{4-\tau}\int_{B_{\rho}(\overline{x}_k)}K\lambda_k^{4-\tau}\frac{\de}{\de\lambda_k}(U^{4-\tau}_{p_k}\circ\delta_{\lambda_k}\chi_{\overline{x}_k}) +$$
 $$- a_k\lambda_k^{-4}\int_{B_{\lambda_k\rho}(\overline{x}_k)}(K(p_k)+O(|x|))\circ\delta_{\lambda_k^{-1}}\left(a_k\frac{1}{\lambda_k}A_{p_k}+\sum_{j\ne k}a_j\frac{1}{\lambda_j}G_{\overline{x}_j}(p_k) + O(|x|)\right)^{3-\tau}\circ\delta_{\lambda_k^{-1}}U_{p_k}\chi_{\overline{x}_k}\circ\delta_{\lambda_k^{-1}}dV+$$
 $$- a_k\int_{B_{\rho}(\overline{x}_k)}K(\widetilde{\Psi}_{\bm{p},\bm{\lambda},\bm{a}}-\Psi_{\bm{p},\bm{\lambda},\bm{a}})^{3-\tau}\Xi U_{p_k}\circ\delta_{\lambda_k}\chi_{\overline{x}_k} =$$
 $$=O(\tau^2)+ 3a_k^3 \left(a_k\frac{1}{\lambda_k^3}A_{p_k}+\sum_{j\ne k}a_j\frac{1}{\lambda_j\lambda_k^{2}}G_{\overline{x}_j}(p_k)\right)\left(\int_{\H^1}U^3 +o(1)\right)+$$
 $$+O\left(\lambda_k^{-2}\tau^{1/2}\int_{B_{\lambda_k\rho}(\overline{x}_k)}|x|\circ\delta_{\lambda_k^{-1}}\cdot (1+|x\cdot p_k^{-1}|^2)^3\chi_{\overline{x}_k}\right) +$$
 $$+ O\left(\tau^{1/2}\lambda_k^2\lambda_k^{-4}\int_{B_{\lambda_k\rho}(\overline{x}_k)}\Xi(U_{p_k}^3)\right) +$$
 $$- a_k^{4-\tau}\lambda_k^{3-\tau}\lambda_k^{-4}\int_{B_{\lambda_k\rho}(\overline{x}_k)}K\circ\delta_{\lambda_k^{-1}}U^{4-\tau}_{\delta_{\lambda_k}(p_k)}\chi_{\overline{x}_k}\circ\delta_{\lambda_k^{-1}} +$$
 $$- a_k^{4-\tau}\frac{1}{4-\tau}\int_{B_{\rho}(\overline{x}_k)}K\lambda_k^{3-\tau}\Xi(U^{4-\tau}_{p_k})\circ\delta_{\lambda_k}\chi_{\overline{x}_k} + O(\lambda_k^{-5})+O(\lambda_k^{-5})=$$
 $$=O(\tau^2)+ \left[\frac{24}{K(\overline{x}_k)^{3/2}} \left(\frac{2}{K(\overline{x}_k)^{1/2}}\frac{1}{\lambda_k^3}A_{p_k}+\sum_{j\ne k}\frac{2}{K(\overline{x}_j)^{1/2}}\frac{1}{\lambda_j\lambda_k^3}G_{\overline{x}_j}(p_k)\right)+o(1)\right]\left(\int_{\H^1}U^3 +o(1)\right)+$$
 $$+O(\tau^2)+0+$$
 $$- a_k^{4-\tau}\lambda_k^{-1-\tau}\int_{B_{\lambda_k\rho}(\overline{x}_k)}K\circ\delta_{\lambda_k^{-1}}U^{4-\tau}_{\delta_{\lambda_k}(p_k)}\chi_{\overline{x}_k}\circ\delta_{\lambda_k^{-1}} +$$
 $$ +a_k^{4-\tau}\frac{1}{4-\tau}\lambda_k^{3-\tau}\int_{B_{\rho}(\overline{x}_k)}\Xi KU^{4-\tau}_{p_k}\circ\delta_{\lambda_k}\chi_{\overline{x}_k} +$$
 $$+a_k^{4-\tau}\frac{1}{4-\tau}\lambda_k^{3-\tau}\int_{B_{\rho}(\overline{x}_k)}KU^{4-\tau}_{p_k}\circ\delta_{\lambda_k}\Xi\chi_{\overline{x}_k}+$$
 $$+ a_k^{4-\tau}\frac{1}{4-\tau}\lambda_k^{3-\tau}\int_{B_{\rho}(\overline{x}_k)}KU^{4-\tau}_{p_k}\circ\delta_{\lambda_k}\chi_{\overline{x}_k}\operatorname{div}\Xi=$$
 $$=O(\tau^2)+ \left[\frac{24}{K(\overline{x}_k)^{3/2}} \left(\frac{2}{K(\overline{x}_k)^{1/2}}\frac{1}{\lambda_k^3}A_{p_k}+\sum_{j\ne k}\frac{2}{K(\overline{x}_j)^{1/2}}\frac{1}{\lambda_j\lambda_k^3}G_{\overline{x}_j}(p_k)\right)+o(1)\right]\left(\int_{\H^1}U^3 +o(1)\right)+$$
 $$- a_k^{4-\tau}\lambda_k^{-1-\tau}\int_{B_{\lambda_k\rho}(\overline{x}_k)}K\circ\delta_{\lambda_k^{-1}}U^{4-\tau}_{\delta_{\lambda_k}(p_k)}\chi_{\overline{x}_k}\circ\delta_{\lambda_k^{-1}} +$$
 $$+a_k^{4-\tau}\frac{1}{4-\tau}\lambda_k^{3-\tau}\int_{B_1(\overline{x}^k)}(x\cerchio{X} +y\cerchio{Y}+2t\cerchio{T})\left[K(\overline{x}^k) + x(XK)(\overline{x}^k)+\right.$$
 $$+ y(YK)(\overline{x}^k) +xy(XYK)(\overline{x}^k) + \frac{1}{2}x^2(X^2K)(\overline{x}^k) + \frac{1}{2}y^2(Y^2K)(\overline{x}^k) + t(TK)(\overline{x}^k)+$$
 $$\left.+O(|x|^3)\right]U^{4-\tau}_{p_k}\circ\delta_{\lambda_k}\chi_{\overline{x}_k} +$$
 $$+O\left(\lambda_k^{3-\tau}\int_{B_{\rho}(\overline{x}_k)\setminus B_{\rho/2}(\overline{x}_k)}\lambda_k^{-(4-\tau)}\frac{1}{|x|^{4-\tau}}\right)+$$
 $$+ a_k^{4-\tau}\frac{1}{4}\left(1+\frac{\tau}{4}+O(\tau^2)\right)\lambda_k^{3-\tau}\lambda_k^{-4}\int_{B_{\lambda_k\rho}(\overline{x}_k)}K\circ\delta_{\lambda_k^{-1}}U^{4-\tau}_{p_k}\chi_{\overline{x}_k}\circ\delta_{\lambda_k^{-1}}(4+O(|x|))\circ\delta_{\lambda_k^{-1}}=$$
 $$=O(\tau^2)+ \left[\frac{24}{K(\overline{x}_k)^{3/2}} \left(\frac{2}{K(\overline{x}_k)^{1/2}}\frac{1}{\lambda_k^3}A_{p_k}+\sum_{j\ne k}\frac{2}{K(\overline{x}_j)^{1/2}}\frac{1}{\lambda_j\lambda_k^3}G_{\overline{x}_j}(p_k)\right)+o(1)\right]\left(\int_{\H^1}U^3 +o(1)\right)+$$
 $$+a_k^{4-\tau}\frac{1}{4-\tau}\lambda_k^{3-\tau}\int_{B_1(\overline{x}^k)}\left[x(XK)(\overline{x}^k)+xy(XYK)(\overline{x}^k) + x^2(X^2K)(\overline{x}^k) +\right.$$
 $$\left.+y(YK)(\overline{x}^k)+ xy(XYK)(\overline{x}^k) +y^2(Y^2K)(\overline{x}^k)+2t(TK)(\overline{x}^k) +O(|x|^3)\right]U^{4-\tau}_{p_k}\circ\delta_{\lambda_k}\chi_{\overline{x}_k} +$$
 $$+O(\lambda_k^{-7})+$$
 $$+ a_k^{4-\tau}\left(\frac{\tau}{4}+O(\tau^2)\right)\lambda_k^{-1-\tau}\left(K(\overline{x}^k)\int_{\H^1}U^4 + o(1) \right)+$$
 $$+ O\left(\lambda_k^{-1-\tau}\int_{B_{\lambda_k\rho}(\overline{x}_k)}U^{4-\tau}_{p_k}O(|x|)\circ\delta_{\lambda_k^{-1}}\right)=$$
 $$=O(\tau^2)+ \left[\frac{24}{K(\overline{x}_k)^{3/2}} \left(\frac{2}{K(\overline{x}_k)^{1/2}}\frac{1}{\lambda_k^3}A_{p_k}+\sum_{j\ne k}\frac{2}{K(\overline{x}_j)^{1/2}}\frac{1}{\lambda_j\lambda_k^3}G_{\overline{x}_j}(p_k)\right)+o(1)\right]\left(\int_{\H^1}U^3 +o(1)\right)+$$
 $$+a_k^{4-\tau}\frac{1}{4-\tau}\lambda_k^{3-\tau}\int_{B_1(\overline{x}^k)}\left[x^2(X^2K)(\overline{x}^k) +y^2(Y^2K)(\overline{x}^k)+O(|x|^3)\right]U^{4-\tau}_{p_k}\circ\delta_{\lambda_k}\chi_{\overline{x}_k} +$$
 $$+ \left(\frac{4}{K(\overline{x}_k)^2}+o(1)\right)\left(\frac{\tau}{4}+O(\tau^2)\right)\lambda_k^{-1-\tau}\left(K(\overline{x}^k)\int_{\H^1}U^4 + o(1) \right)+ O(\tau^2)=$$
 $$=O(\tau^2)+ \left[\frac{6}{K(\overline{x}_k)^{1/2}} \left(\frac{2}{K(\overline{x}_k)^{1/2}}\frac{1}{\lambda_k^3}A_{p_k}+\sum_{j\ne k}\frac{2}{K(\overline{x}_j)^{1/2}}\frac{1}{\lambda_j\lambda_k^{2}}G_{\overline{x}_j}(p_k)\right)\right]\left(\int_{\H^1}U^3 +o(1)\right)+$$
 $$+\left(\frac{1}{K(\overline{x}_k)^2}+o(1)\right)\lambda_k^{-3-\tau}\frac{1}{2}(1\cdot\Delta_bK(\overline{x}^k))\int_{\H^1}(x^2+y^2)U^4 +$$
 $$+ \tau\left(\frac{1}{K(\overline{x}_k)^2}+o(1)\right)\lambda_k^{-1}\left(K(\overline{x}^k)\int_{\H^1}U^4 + o(1) \right)=$$
 $$= -\frac{1}{\lambda_k^2}\sum_{j=1}^NM_{kj}(S)\frac{1}{\lambda_j} + \frac{1}{K(\overline{x}_k)}\frac{\tau}{\lambda_k}\int_{\H^1}U^4 +o(\tau^3)$$
 The function
 $$(\lambda_1,\ldots,\lambda_N)\mapsto -\frac{1}{\lambda_k^2}\sum_{j=1}^NM_{kj}(S)\frac{1}{\lambda_j} + \frac{\pi^2}{4K(\overline{x}_k)}\frac{\tau}{\lambda_k}\in\R^N$$
 is the gradient of $F$ defined as
 $$F(\lambda_1,\ldots,\lambda_N) = \frac{1}{2}\sum_{j,k=1}^NM_{kj}(S)\frac{1}{\lambda_j\lambda_k} + \sum_{k=1}^N\frac{\pi^2}{4K(\overline{x}_k)}\tau\log\lambda_k$$
 which is strictly convex, with strictly definite second differential, and such that $\nabla F$ has pexactly one zero, with components of order $\sim\tau^{-1/2}$.
 Therefore, up to changing the constant $A$ in the definition of $V'_{\tau}(S)$, we have proved that
 \begin{equation}\label{GradoDimostrazione3}
  \deg(\nabla_{\bm{\lambda}}\ci{J}_{\tau},\sum_{k=1}^N\Omega_{N+k},0) = 1.
 \end{equation}
 Finally Proposition \ref{DifferenzialeSecondoPositivo} implies that
 \begin{equation}\label{GradoDimostrazione4}
  \deg(\nabla_{v}\ci{J}_{\tau},\Omega_{3N+1},0) = 1.
 \end{equation}
 Putting all together formulas \eqref{GradoDimostrazione1}, \eqref{GradoDimostrazione2}, \eqref{GradoDimostrazione3} and \eqref{GradoDimostrazione4} we can apply Proposition \ref{GradoProdotto} to get the thesis.
\end{proof}

Finally we can prove Theorem \ref{Teorema2}.

\begin{proof}[Proof of Theorem \ref{Teorema2}]
 By Proposition \ref{DicotomiaSoluzioni}, applied with $V'_{\tau}(S)$ defined in formula \eqref{DefinizioneVtauprimo} instead of $U_{\tau}(S)$ by Lemma \ref{Lemmaa}, there exist $\tau_0>0$ and $R$ such that for $\tau\in(0,\tau_0]$ every zero of $\nabla\ci{J}_{\tau}$ is contaided in
 $$\Omega_R\cup\bigcup_{S\subseteq\operatorname{Crit}(K):M(S)>0}U_{\tau}(S).$$
 By Lemma \ref{LimitatezzaSoluzioniSottocritiche} there exists $C$ such that for $\tau=\tau_0$ every such solution is contained in $\Omega_C$.
 By Proposition \ref{ProposizioneGradoSottocritico} and notorious properties of the Leray-Schauder degree (properties (P.7) and (P.3) in Chapter 3 in \cite{AM}) it holds that
 $$-1=\deg(F_{3-\tau_0},\Omega_C,0)= \deg\left(F_{3-\tau_0},\Omega_R\cup\bigcup_{S\subseteq\operatorname{Crit}(K):M(S)>0}V'_{\tau_0}(S),0\right)=$$
 $$= \deg(F_{3-\tau_0},\Omega_R,0) + \sum_{S\subseteq\operatorname{Crit}(K):M(S)>0}\deg(F_{3-\tau_0},V'_{\tau_0}(S),0)$$
 so the thesis follows by the homotopy invariance property of the Leray-Schauder degree and Proposition \ref{GradoPuntiCriticiInfinito}.
\end{proof}

\appendix
\section{A lemma in algebraic topology}
\begin{proposizione}
 Let $\Omega_1\subset\R^{n_1},\ldots,\Omega_k\subset\R^{n_k}$ be contractible open sets and
 $$F_i:\Omega_1\times\cdots\times\Omega_k\to\R^{n_i}$$
 for $i=1,\ldots,k$ satisfying $F_i\ne 0$ on $\Omega_1\times\cdots\times\de\Omega_i\times\ldots\times\Omega_k$.
 Define $\bm{x}=(x_1,\ldots,x_k)\in\Omega_1\times\cdots\times\Omega_k$,
 $$\Omega_i^{\bm{x}}=\{x_1\}\times\ldots\times\Omega_i\times\ldots\times\{x_k\},$$
 $\bm{F}=(F_1,\ldots,F_k)$, $\bm{\Omega}=\Omega_1\times\cdots\times\Omega_k$.
 Then, given $p_i\notin F_i(\Omega_1\times\cdots\times\Omega_k)$ and defining $\bm{p}=(p_1,\ldots,p_k)$, $\deg(F_i,\Omega_i^{\bm{x}},p_i)$ is independent by $\bm{x}$, and
 $$\deg(\bm{F},\bm{\Omega},\bm{p}) = \prod_{i=1}^k\deg(F_i,\Omega_i^{\bm{x}},p_i)$$
\end{proposizione}

\begin{proof}
 By invariance by admissible homotopy, we can assume that $F_i$ depends only on $x_i$, and therefore, by abuse of notation, that $\bm{F}(\bm{x})=(F_1(x_1),\ldots,F_k(x_k))$.
 Calling $\omega_i\in H_{n_i}(\Omega_i,\de\Omega_i)$, $\bm{\omega}\in H_n(\bm{\Omega},\de\bm{\Omega})$, $\eta_i\in H_{n_i}(\R^{n_i},\R^{n_i}\setminus\{p_i\})$ and $\bm{\eta}\in H_n(\R^n,\R^n\setminus\{\bm{p}\})$ (where $n=n_1+\ldots+n_k$) the respective fundamental classes, it is known that an equivalent definition of the degree is
 $$(F_i)_*\omega_i = \deg(F_i,\Omega_i^{\bm{x}},p_i)\eta_i,$$
 $$(\bm{F})_*\bm{\omega} = \deg(\bm{F},\bm{\Omega},\bm{p})\bm{\eta}.$$
 Then the thesis follows from the naturality of the Künneth isomorphism.
\end{proof}

We point out that an alternative proof of the above Proposition can be obtained by using the definition of the degree as a sum of signs on the counterimages of a regular point, and an elementary computation.

\begin{proposizione}\label{GradoProdotto}
 Let $E_1,\ldots,E_k$ be Banach spaces, $\Omega_1\subset E_1,\ldots,\Omega_k\subset E_k$ be contractible open sets, and
 $$G_i:\Omega_1\times\cdots\times\Omega_k\to E_i$$
 for $i=1,\ldots,k$ be compact maps (such that $I-G_i\ne 0$) on $\Omega_1\times\cdots\times\de\Omega_i\times\ldots\times\Omega_k$.
 Define $\bm{x}=(x_1,\ldots,x_k)\in\Omega_1\times\cdots\times\Omega_k$,
 $$\Omega_i^{\bm{x}}=\{x_1\}\times\ldots\times\Omega_i\times\ldots\times\{x_k\},$$
 $\bm{G}=(G_1,\ldots,G_k)$, $\bm{\Omega}=\Omega_1\times\cdots\times\Omega_k$.
 Then, given $p_i\notin (I-G_i)(\Omega_1\times\cdots\times\Omega_k)$ and defining $\bm{p}=(p_1,\ldots,p_k)$, $\deg(F_i,\Omega_i^{\bm{x}},p_i)$ is independent by $\bm{x}$, and
 $$\deg(I-\bm{G},\bm{\Omega},\bm{p}) = \prod_{i=1}^k\deg(I-G_i,\Omega_i^{\bm{x}},p_i)$$
\end{proposizione}

\begin{proof}
 The construction that leads to the definition of the Leray-Schauder degree (see \cite{AM}, section 3.4) uses finite dimensional approximations. It can be easily checked that applying the former theorem to it allows to get the thesis.
\end{proof}

\section{Computations}
\begin{lemma}\label{IntegraleU3}
 The following equalities hold:
 $$\int_{\H^1}U^3= 2\pi,$$
 $$\int_{\H^1}(x^2+y^2)U^4 = \frac{\pi^2}{4},$$
 $$\int_{\H^1}U^4 = \frac{\pi^2}{4}.$$
\end{lemma}

\begin{lemma}\label{StimaDifferenzaSublaplaciano}
 In the hypotheses of Proposition \ref{GradienteNullo}
 $$\int_{B_{\rho}(x_i)}(\Delta_bu_i-\cerchio{\Delta}_bu_i)\cerchio{Z}_ru_id\cerchio{V}=O(M_i^{-1}).$$
\end{lemma}

\begin{proof}
 Using Lemma A.5 in \cite{Af},
 $$\int_{B_{\rho}(x_i)}(\Delta_bu_i-\cerchio{\Delta}_bu_i)\cerchio{Z}_ru_id\cerchio{V}=$$
 $$= O\left(\int_{B_{\rho}(x_i)}\left(|x||(u_i)_{,1}|+ |x||(u_i)_{,\con{1}}| + |x|^2|(u_i)_{,0}|+ |x|^2|(u_i)_{,1\con{1}}| +\right.\right.$$
 $$\left.\left.+ |x|^2|(u_i)_{,11}| +|x|^2|(u_i)_{,\con{1}\con{1}}|+ |x|^3|(u_i)_{,01}| + |x|^3|(u_i)_{,0\con{1}}| +|x|^6|(u_i)_{,00}|\right)\cerchio{Z}_ru_i\right)=$$
 $$=O\left(M_i^{-2(p_i-1)}\int_{B_{R_i}(x_i)}\left(M_i^{-\frac{p_i-1}{2}}|x|M_i^{\frac{p_i-1}{2}}\left(u_i\circ\delta_{M_i^{-\frac{p_i-1}{2}}}\right)_{,1} +\right.\right.$$
 $$+M_i^{-(p_i-1)}|x|^2M_i^{p_i-1}\left(u_i\circ\delta_{M_i^{-\frac{p_i-1}{2}}}\right)_{,0}+M_i^{-3\frac{p_i-1}{2}}|x|^3M_i^{3\frac{p_i-1}{2}}\left(u_i\circ\delta_{M_i^{-\frac{p_i-1}{2}}}\right)_{,01} +$$
 $$\left.\left.+  M_i^{-3(p_i-1)}|x|^6M_i^{2(p_i-1)}\left(u_i\circ\delta_{M_i^{-\frac{p_i-1}{2}}}\right)_{,00}\right)M_i^{\frac{p_i-1}{2}}\cerchio{Z}_r\left(u_i\circ\delta_{M_i^{-\frac{p_i-1}{2}}}\right)\right) + $$
 $$+O\left(\int_{B_{\rho}(x_i)\setminus B_{r_i}(x_i)}\left(|x|M_i^{-1}|x|^{-3} + |x|^2M_i^{-1}|x|^{-4} +|x|^3M_i^{-1}|x|^{-5} +|x|^6M_i^{-1}|x|^{-6}\right)M_i^{-1}|x|^{-3}\right)=$$
 $$=O\left(M_i^{-2(p_i-1)}\int_{B_{R_i}(x_i)}\left(|x|M_i(1+|x|)^{-3} + M_i^{-(p_i-1)}|x|^6M_i(1+|x|)^{-6}\right)M_i^{\frac{p_i-1}{2}}M_i(1+|x|)^{-3}\right) + $$
 $$+O\left(M_i^{-2}\int_{B_{\rho}(x_i)\setminus B_{r_i}(x_i)}|x|^{-5}\right) =$$
 $$=O\left(M_i^{-2}\int_{B_{R_i}(x_i)}\left(M_i|x|(1+|x|)^{-6} + M_i^{-1}(1+|x|)^{-3}\right)\right) + O\left(M_i^{-2}r_i^{-1}\right)=$$
 $$=O(M_i^{-1}).$$
\end{proof}

\begin{lemma}\label{LemmauZru}
 In the hypotheses of Proposition \ref{GradienteNullo}
 $$\int_{B_{\rho}(x_i)}Ru_i\cerchio{Z}_ru_i\chi d\cerchio{V}=O\left(M_i^{-1}\right).$$
\end{lemma}

\begin{proof}
 $$\int_{B_{\rho}(x_i)}Ru_i\cerchio{Z}_ru_i\chi d\cerchio{V}=$$
 $$=M_i^{-2(p_i-1)}\int_{B_{R_i}(x_i)}R\circ\delta_{M_i^{-\frac{p_i-1}{2}}}u_i\circ\delta_{M_i^{-\frac{p_i-1}{2}}}M_i\cerchio{Z}_ru_i\circ\delta_{M_i^{-\frac{p_i-1}{2}}}\chi\circ\delta_{M_i^{-\frac{p_i-1}{2}}} +$$
 $$+\int_{B_{\rho}(x_i)\setminus B_{r_i}(x_i)}Ru_i\cerchio{Z}_ru_i\chi d\cerchio{V}=$$
 $$= O\left(M_i^{-2(p_i-1)}M_i\int_{B_{R_i}(x_i)}M_i(1+|x|)^{-2}M_i(1+|x|)^{-3}\right)+$$
 $$+O\left(\int_{B_{\rho}(x_i)\setminus B_{r_i}(x_i)}M_i^{-1}|x|^{-2}M_i^{-1}|x|^{-3}d\cerchio{V}\right)=$$
 $$= O\left(M_i^{-1}\right)+O\left(M_i^{-2}r_i^{-1}\right)=O\left(M_i^{-1}\right).$$
\end{proof}

\begin{lemma}\label{StimaIntegralex1u4}
 If $x_i\to\overline{x}$ is an isolated simple blow up point for the sequence of solutions $u_i$ and if $0\le\alfa<4$ then
 $$\int_{B_{\rho}(x_i)}|x|^{\alfa}u_i^{p_i+1}d\cerchio{V}=O\left(M_i^{-\alfa}\right).$$
\end{lemma}

\begin{proof}
 $$\int_{B_{\rho}(x_i)}|x|^{\alfa}u_i^{p_i+1}d\cerchio{V}=$$
 $$=O\left(M_i^{-2(p_i-1)}\int_{B_{R_i}(x_i)}M_i^{-\frac{p_i-1}{2}\alfa}|x|\left(u_i\circ\delta_{M_i^{-\frac{p_i-1}{2}}}\right)^{p_i+1}d\cerchio{V}+\int_{B_{\rho}(x_i)\setminus B_{r_i}(x_i)}|x|^{\alfa}u_i^{p_i+1}d\cerchio{V}\right)=$$
 $$=O\left(M_i^{-4-\alfa}\int_{B_{R_i}(x_i)}|x|^{\alfa}M_i^{p_i+1}(1+|x|)^{-2(p_i+1)}d\cerchio{V}+\int_{B_{\rho}(x_i)\setminus B_{r_i}(x_i)}|x|^{\alfa}M_i^{-(p_i+1)}|x|^{-2(p_i+1)}d\cerchio{V}\right)=$$
 $$=O\left(M_i^{-\alfa}+M_i^{-4}r_i^{\alfa+4-2(p_i+1)}\right)=O\left(M_i^{-\alfa}\right).$$
\end{proof}

\begin{lemma}\label{StimaIntegralePsi3MenoTau}
 Under the hypotheses of Lemma \ref{LyapunovSchmidt}
 $$\int_M\left(\Psi_{\bm{p},\bm{\lambda},\bm{a}}^{3-\tau}\right)^{\frac{4}{3}} \lesssim 1$$
\end{lemma}

\begin{lemma}\label{LemmaB3}
 Under the hypotheses of Lemma \ref{LyapunovSchmidt}
 $$\left(\int_{B_{\rho}(\overline{x}^k)}\left(\phi_{p_k,\lambda_k}^3-\phi_{p_k,\lambda_k}^{3-\tau}\right)^{\frac{4}{3}}\right)^{\frac{3}{4}}= O(\tau|\log\tau|).$$
\end{lemma}

\begin{proof}
 $$\int_{B_{\rho}(\overline{x}^k)}\left(\phi_{p_k,\lambda_k}^3-\phi_{p_k,\lambda_k}^{3-\tau}\right)^{\frac{4}{3}}=$$
 $$=\lambda_k^{-4}\int_{B_{\lambda_k\rho}(\overline{x}^k)}\left(\lambda_k^3(1+|x\cdot\delta_{\lambda_k}p_k^{-1}|)^{-6}-\lambda_k^{3-\tau}(1+|x\cdot\delta_{\lambda_k}p_k^{-1}|)^{-6+2\tau}\right)^{\frac{4}{3}}\lesssim$$
 $$= O\left(\int_{B_{\lambda_k\rho}(\overline{x}^k)}\left((1+|x\cdot\delta_{\lambda_k}p_k^{-1}|)^{-6}-(1+|x\cdot\delta_{\lambda_k}p_k^{-1}|)^{-6+2\tau}\right)^{\frac{4}{3}}\right)+$$
 $$+O\left((\tau|\log\tau|)^{4/3}\int_{B_{\lambda_k\rho}(\overline{x}^k)}\left((1+|x\cdot\delta_{\lambda_k}p_k^{-1}|)^{-6+2\tau}\right)^{\frac{4}{3}}\right)$$
 Since on the domain of integration $|(1+|x\cdot\delta_{\lambda_k}p_k^{-1}|)^{2\tau}- 1|\lesssim\tau|\log\tau|$, the thesis follows.
\end{proof}

\begin{lemma}\label{LemmaDifferenzaSublaplaciano}
 $$\left(\int_{B_{\rho/2}(\overline{x}^k)}|(\cerchio{\Delta}_b-\Delta_b)\phi_{p_k,\lambda_k}|^{4/3}\right)^{3/4}=O(\tau^{1/2}).$$
\end{lemma}

\begin{proof}
 $$\left(\int_{B_{\rho/2}(\overline{x}^k)}|(\cerchio{\Delta}_b-\Delta_b)\phi_{p_k,\lambda_k}|^{4/3}\right)^{3/4}=$$
 $$+ O\left(\left[\int_{B_{\rho/2}(\overline{x}^k)}\left(|x|^3|(\phi_{p_k,\lambda_k})_{,1}|+ |x|^3|(\phi_{p_k,\lambda_k})_{,\con{1}}| + |x|^4|(\phi_{p_k,\lambda_k})_{,0}| + |x|^4|(\phi_{p_k,\lambda_k})_{,1\con{1}}| +|x|^4|(\phi_{p_k,\lambda_k})_{,11}| +\right.\right.\right.$$
 $$\left.\left.\left.+|x|^4|(\phi_{p_k,\lambda_k})_{,\con{1}\con{1}}|+ |x|^5|(\phi_{p_k,\lambda_k})_{01}| + |x|^5|(\phi_{p_k,\lambda_k})_{,0\con{1}}| +|x|^{10}|(\phi_{p_k,\lambda_k})_{,00}|\right)^{\frac{4}{3}}\right]^{\frac{3}{4}}\right)=$$
 $$= O\left(\left[\lambda_k^{-4}\int_{B_{\rho\lambda_i/2}(\overline{x}^k)}\left(\lambda_k^{-3}|x|^3\lambda_k\lambda_k(1+|x\cdot\delta_{\lambda_k}p_k^{-1}|)^{-3} + \lambda_k^{-10}|x|^{10}\lambda_k\lambda_k^4(1+|x\cdot\delta_{\lambda_k}p_k^{-1}|)^{-6} \right)^{\frac{4}{3}}\right]^{\frac{3}{4}}\right)=$$
 $$= O\left(\left[\lambda_k^{-4}\int_{B_{\rho\lambda_i/2}(\overline{x}^k)}\left(\lambda_k^{-1} + \lambda_k^{-5}(1+|x|)^2\right)^{\frac{4}{3}}\right]^{\frac{3}{4}}\right)=$$
 $$= O\left(\sum_{k=1}^N\lambda_k^{-3}\lambda_k^{-1}\left[\Vol(B_{\rho\lambda_i/2}(\overline{x}^k))\right]^{\frac{3}{4}}\right)= O(\tau^{1/2}).$$
\end{proof}


\textsc{Claudio Afeltra, Department of Mathematics, University of Trento, Via Sommarive 14, 38123 Povo (Trento), Italy}
 
 \textit{Email address}:  \texttt{claudio.afeltra@unitn.it}
 
\end{document}